%% file: main.tex
\documentclass[11pt,reqno,twoside]{article}
\usepackage{mysty}

\newcommand{\textoperatorname}[1]{%
  \operatorname{\textnormal{#1}}%
}

\usepackage{wrapfig}
\usepackage{mdwlist}

\usepackage{fullpage}
\usepackage{patchcmd,pifont}


\usepackage{enumitem}

\SetLabelAlign{parright}{\parbox[t]{\labelwidth}{\raggedleft{#1}}}
\setlist[description]{style=multiline,topsep=4pt,align=parright}

\makeatletter
\let\reftagform@=\tagform@
\def\tagform@#1{\maketag@@@{(\ignorespaces\textcolor{black}{#1}\unskip\@@italiccorr)}}
\newcommand{\iref}[1]{\textup{\reftagform@{\tcr{\ref{#1}}}}}
\makeatother

\begin{document}
\title{An SDE perspective on stochastic convex optimization}
\author{Rodrigo Maulen S.\thanks{Normandie Universit\'e, ENSICAEN, UNICAEN, CNRS, GREYC, France. E-mail: rodrigo.maulen@ensicaen.fr} \and
Jalal Fadili\thanks{Normandie Universit\'e, ENSICAEN, UNICAEN, CNRS, GREYC, France. E-mail: Jalal.Fadili@ensicaen.fr} \and Hedy Attouch\thanks{IMAG, CNRS, Universit\'e Montpellier, France. E-mail: hedy.attouch@umontpellier.fr}
}
\date{}
\maketitle

\begin{abstract}
In this paper, we analyze the global and local behavior of gradient-like flows under stochastic errors towards the aim of solving convex optimization problems with noisy gradient input. We first study the unconstrained differentiable convex case, using a stochastic differential equation where the drift term is minus the gradient of the objective function and the diffusion term is either bounded or square-integrable. In this context, under Lipschitz continuity of the gradient, our first main result shows almost sure convergence of the objective and the trajectory process towards a minimizer of the objective function. We also provide a comprehensive complexity analysis by establishing several new pointwise and ergodic convergence rates in expectation for the convex, strongly convex and (local) {\L}ojasiewicz case. The latter, which involves local analysis, is  challenging and requires non-trivial arguments from measure theory. Then, we extend our study to the constrained case and more generally to certain nonsmooth situations. We show that several of our results have natural extensions obtained by replacing the gradient of the objective function by a cocoercive monotone operator. This makes it possible to obtain similar convergence results for optimization problems with an additively  "smooth + non-smooth" convex structure. Finally, we consider another extension of our results to non-smooth optimization which is based on the Moreau envelope.

\end{abstract}

\begin{keywords}
Convex optimization, Stochastic Differential Equation, Stochastic gradient descent, {\L}ojasiewicz inequality, KL inequality, Convergence rate, Asymptotic behavior.
\end{keywords}

\begin{AMS}
37N40, 46N10, 49M30, 65B99, 65K05, 65K10, 90B50, 90C25
\end{AMS}


\input{tex/sec-Intro}
\input{tex/sec-Notation}
\input{tex/sec-Smooth}

\input{tex/Nonsmooth_structured}
\input{tex/sec-Appendix}

\bibliographystyle{unsrt}
\smaller
\bibliography{citas}

\end{document}

%% file: tex/sec-Intro.tex
\section{Introduction}\label{sec:intro}
\subsection{Problem Statement}\label{problem_statement}

We aim to solve convex minimization problems by means of stochastic differential equations whose drift term is driven by the gradient of the objective function. This allows for noisy (inaccurate) gradient input to be taken into account. Consider the minimization problem
\begin{equation}\label{P}\tag{P}
    \min_{x\in \R^d} f(x),
\end{equation}

where the objective $f$ satisfies the following standing assumptions:
\begin{align}\label{H0}
\begin{cases}
\text{$f$ is continuously differentiable and convex with $L$-Lipschitz continuous gradient}; \\
\calS \eqdef \argmin (f)\neq\emptyset. \tag{$\mathrm{H}_0$} 
\end{cases}
\end{align} 
We will also later deal with the constrained case, and more generally with additively structured "smooth + nonsmooth" convex optimization.

Let us first recall some basic facts about the deterministic case. To solve \eqref{P}, a fundamental dynamic to consider is the gradient flow of $f$, \ie the gradient descent dynamic with initial condition $X_0\in\R^d$:
\begin{equation}
\begin{cases}\label{GF}\tag{GF}
\begin{aligned}
\dot{x}&=-\nabla f(x),\quad t>0\\
x(0)&=X_0.
\end{aligned}
\end{cases}
\end{equation}

It is well known since the founding papers of Brezis, Baillon, Bruck in the 1970s that, if the solution set $\argmin f$ of \eqref{P} is non-empty, then each solution trajectory of \eqref{GF} converges, and its limit belongs to $\argmin f$.
In fact, this result is true in a more general setting, simply assuming that the objective function $f$ is convex, lower semicontinuous (lsc) and proper (in which case we must consider the differential inclusion obtained by replacing in \eqref{GF} the gradient of $f$ by the sub-differential $\partial f$).

In many cases, the gradient input is subject to noise, for example, if the gradient cannot be evaluated directly, or due to some other exogenous factor. In such scenario, one can model these errors using a stochastic integral with respect to the measure defined by a continuous It\^o martingale. This entails the following stochastic differential equation as a stochastic counterpart of \eqref{GF}:
\begin{equation}\label{CSGD}\tag{$\mathrm{SDE}$}
\begin{cases}
\begin{aligned}
dX(t)&=-\nabla f(X(t))dt+\sigma(t,X(t))dW(t), \quad t> 0\\
X(0)&=X_0,
\end{aligned}
\end{cases}
\end{equation}
defined over a filtered probability space $(\Omega,\calF,\{\calF_t\}_{t\geq 0},\PP)$, where the diffusion (volatility) term $\sigma:\R_+\times\R^d\rightarrow \R^{d\times m}$ is matrix-valued measurable function, and $W$ is the $m$-dimensional Brownian motion. 

Our goal is to study the dynamic of \eqref{CSGD} and its long time behavior in order to solve \eqref{P}. To identify the assumptions necessary to hope for such a behavior to occur, remember that when the diffusion term $\sigma$ is a positive real constant, it is well-known that $X(t)$ in this case is a continuous-time diffusion process known as Langevin diffusion, and has a unique invariant probability measure $\pi_{\sigma}$ with density $\propto e^{-2 f(x)/\sigma^2}$ \cite{Bhattacharya78}. It is easy to see that the measure $\pi_\sigma$ gets concentrated around $\argmin f$ as $\sigma$ tends to $0^+$ with $\lim_{\sigma \to 0^+} \pi_\sigma(\argmin f) = 1$; see \eg \cite{Catoni99}.

Motivated by this observation, our paper will then mostly focus on the case where $\sigma(\cdot,x)$ vanishes sufficiently fast as $t \to +\infty$ uniformly in $x$, and some guarantees will also be provided for uniformly bounded $\sigma$. Therefore, throughout, the entries $\sigma_{ik}$ are assumed to satisfy:
\begin{equation*}
\tag{$\mathrm{H}$}\label{H}
\begin{cases}
\sup_{t \geq 0,x \in \R^d} |\sigma_{ik}(t,x)|<\infty, \\
|\sigma_{ik}(t,x')-\sigma_{ik}(t,x)|\leq l_0\norm{x'-x},
\end{cases}
\end{equation*}   
for some $l_0>0$ and for all  $t\geq 0, x, x'\in \R^d$. The Lipschitz continuity assumption is mild and required to ensure well-posedness of \eqref{CSGD}.



\subsection{Contributions}

We study the properties of the process $X(t)$ and $f(X(t))$ for the stochastic differential equation \eqref{CSGD} from an optimization perspective, under the assumptions \eqref{H0} and \eqref{H}. When the diffusion term is uniformly bounded, we show convergence of $\EE[f(X(t))-\min f]$ to a noise-dominated region both for the convex and strongly convex case. When the diffusion term is square-integrable, we show in Theorem \ref{converge2} that $X(t)$ converges almost surely to a solution of \eqref{P}, which is a new result to the best of our knowledge. Moreover, in Theorem \ref{importante0} and Proposition~\ref{beta}, we provide new ergodic and pointwise convergence rates of the objective in expectation, again for both the convex and strongly convex case. 

Then we turn to a local analysis relying on the {\L}ojasiewicz inequality and its strong ties with error bounds. Since this property is most often satisfied only locally, we deepen the discussion on the long time localization of the process. This is fundamental, because in the recent literature on local convergence properties of stochastic gradient descent, strong assumptions are imposed, such as $X(t)$ or $f(X(t))$ is locally bounded almost surely. Such assumptions are unfortunately unrealistic due to the presence of the Brownian Motion. We manage to circumvent this problem by using arguments from measure theory, in particular Egorov's theorem. In turn, under the {\L}ojasiewicz inequality assumption with exponent $q \geq 1/2$, this allows us to show local convergence rates of the objective and the trajectory itself in expectation over a set of events whose probability is arbitrarily close to $1$ (see Theorem \ref{importante3}).

Table~\ref{table:summary} summarizes the local and global convergence rates obtained for $\EE[f(X(t))-\min f]$. In this table, $\delta > 0$ is a parameter which is intended to be taken arbitrarily close to $0$ but different from it, $\sigma_*>0$ and $\sigma_{\infty}(\cdot)$ are defined as
\begin{equation}\label{eq:defsigstar}
\norm{\sigma(t,x)}_F^2\leq \sigma_*^2, \quad \forall t\geq 0,\forall x\in\R^d, \qqandqq 
\sigma_{\infty}(t)\eqdef \sup_{x\in\R^d}\norm{\sigma(t,x)}_F ,
\end{equation}
and $\sigma_{\infty}(\cdot)$ is a decreasing function. $\Loj^q(\calS)$ is the class of functions satisfying the {\L}ojasiewicz inequality with exponent $q \in [0,1]$ at each point of $\calS$ (see Definition~\ref{def:lojq})\footnote{Semialgebraic and more generally analytic functions is a typical class verifying the {\L}ojasiewicz inequality at each point \cite{loj1,loj2}.}.

\begin{table}[H]

\begin{center}
\begin{tabular}{|c|c|c|cc|}
\hline
\textbf{Property of $f$} & \textbf{Gradient Flow} & \textbf{SDE $(\sup_{t\geq 0}\sigma_{\infty}(t)\leq\sigma_*)$} & \multicolumn{2}{c|}{\textbf{SDE $(\sigma_{\infty}\in \Lp^2(\R_+))$}}    \\ \hline
Convex & $t^{-1}$ & $t^{-1}+\sigma_*^2$ & \multicolumn{2}{c|}{$t^{-1}$}   \\ \hline
$\mu-$Strongly Convex & $e^{-2\mu t}$ & $e^{-2\mu t}+\sigma_*^2$ & \multicolumn{2}{c|}{
$\max\{e^{-2\mu t},\sigma_{\infty}^2(t)\}$} \\ \hline
Convex $\cap$ $\Loj^{1/2}(\calS)$ (coef. $\mu$) & $e^{-\mu^2 t}$ & \ding{56} & \multicolumn{2}{c|}{
$\max\{e^{-\mu^2 t},\sigma_{\infty}^2(t)\}+\sqrt{\delta}$} \\ \hline
Convex $\cap$ $\Loj^q(\calS)$, $q\in (\frac{1}{2},1)$ & $t^{-\frac{1}{2q-1}}$ & \ding{56} & \multicolumn{2}{c|}{
$t^{-\frac{1}{2q-1}}$ \tablefootnote{This is not yet proven, our conjecture is that it is true when $\sigma_{\infty}=\calO((t+1)^{-\frac{q}{2q-1}})$ (see the detailed discussion in Conjecture~\ref{conjecture}).}$+\sqrt{\delta}$}   \\ \hline
\end{tabular}
\caption{Summary of local and global convergence rates obtained for $\EE[f(X(t))-\min f]$.}
\label{table:summary}
\end{center}
\end{table}
Although it is natural to think that we can take the limit when $\delta$ goes to $0^+$, the time from which these convergence rates are valid depends on $\delta$ and increases (potentially to $+\infty$) as $\delta$ approaches $0^+$. Assuming only the boundedness of the diffusion and the {\L}ojasiewicz inequality, we could not find better results (cells marked with \ding{56}) than those presented in the convex case. Since the {\L}ojasiewicz inequality is local, a natural approach would be to localize the process in the long term with high probability. However, it is not clear how to achieve this. \\ 




In Section~\ref{nonsmooth_structured}, we turn to extending some of the preceding results to the structured convex minimization problem 
\begin{equation}\tag{$\mathrm{P_c}$}\label{P0}
    \min_{x\in\R^d} f(x)+g(x),
\end{equation}
where $f:\R^d\rightarrow\R$ satisfies \eqref{H0}, $g: \R^d\rightarrow \Rinf$ is proper, lsc and convex and $\argmin (f+g) \neq \emptyset$. This obviously covers the case of constrained minimization of $f$ over a non-empty closed convex set. We take two different routes leading to different SDEs.

The first approach consists in reformulating \eqref{P0} as finding for zeros of the operator $M_{\mu}: \R^d \to \R^d$ 
\[
M_{\mu}(x)=\frac{1}{\mu} \pa{x- \prox_{\mu g}( x-\mu \nabla f(x))} ,
\]
where $\mu > 0$ and $\prox_{\mu g}$ is the proximal mapping of $\mu g$. It is well-known that the operator $M_{\mu}$ is cocoercive \cite{cocoercive}, hence monotone and Lipschitz continuous, and $M_{\mu}=\nabla f$ when $g$ vanishes. The idea is then to replace the operator $\nabla f$ in \eqref{CSGD} by $M_{\mu}$ leading to an SDE which will have many of the convergence properties obtained in the smooth convex case. This approach is in accordance with the deterministic theory for monotone cocoercive operators (see \cite{contgradp,minimiz,cocoercive}).  

The second approach regularizes the nonsmooth component $g$ of the objective function using its Moreau envelope 
\[
g_{\theta} (x) = \min_{z \in \R^d} g(z) + \frac{1}{2 \theta} \norm{x - z}^2 .
\]
This leads to studying the dynamic \eqref{CSGD} with the function $f+g_{\theta}$, which has a continuous Lipschitz gradient. This approximation method leads to a non-autonomous SDE. Note, however, that the noise in this case can be considered on the evaluation of $\nabla f(x)$, while it is on $M_{\mu}(x)$ in the first approach.

\subsection{Relation to prior work}
The gradient system \eqref{GF}, which is valid on a general real Hilbert space $\mathcal{H}$, is a dissipative dynamical system, whose study dates back to Cauchy \cite{Cauchy}. It plays a fundamental role in optimization: it transforms the problem of minimizing $f$ into the study of the asymptotic behavior of the trajectories of \eqref{GF}. This example was the precursor to the rich connection between continuous dissipative dynamical systems and optimization. Its Euler forward discretization (with stepsize $\gamma_k>0$) is the celebrated gradient descent scheme 
\begin{equation}\label{GD}\tag{GD}
    x_{k+1}=x_k-\gamma_k \nabla f(x_k).
\end{equation} 
Under \eqref{H0}, and for $(\gamma_k)_{k \in \N} \subset ]0,2/L[$, then we have both the convergence of the values $f(x_k)- \min f=\calO(1/k)$ (in fact even $o(1/k)$), and the weak convergence of iterates $(x_k)_{k \in\N} $ to a point in $\argmin f$. Moreover, if the {\L}ojasiewicz inequality \eqref{li} (see \cite{loj3}) is satisfied, then we can ensure the strong convergence of $(x_k)_{k\in\N}$ to a point in $\argmin f$ and faster convergence rates than those ensured by the simple convexity hypothesis (see \cite{BolteKLComplexity16,applications}).

\smallskip

Now, let us focus on the finite-dimensional case ($\mathcal{H}=\R^d$). Although the Gradient Descent is a classical algorithm to solve the convex minimization problem, with the need to handle large-scale problems (such as in various areas of data science and machine learning), there has become necessary to find ways to get around the high computational cost per iteration that these problems entail. The Robbins-Monro stochastic approximation algorithm \cite{rob} is at the heart of Stochastic Gradient Descent (SGD), which, roughly speaking, consists in cheaply and randomly approximating the gradient at the price of obtaining a random noise in the solutions. 
Given an initial point $x_0\in\R^d$, (SGD) updates the iterates according to 
\begin{equation}\tag{SGD}\label{sgdd}
x_{k+1}=x_k-\gamma\nabla f(x_k)-\gamma\xi_k, 
\end{equation}
where $\xi_k$ denotes the (random) noise term at the $k-$th iteration.




\smallskip

Recent work (see \cite{optch,continuous,dif,schrodinger,valid, cycle}) has linked algorithm \eqref{sgdd} with dynamic \eqref{CSGD}, showing the context under which \eqref{CSGD} can be seen as an approximation (under a specific error) of \eqref{sgdd} and vice-versa. For example, \eqref{CSGD} can be interpreted as the pathwise solution to the Fokker-Planck equation (see \cite{Gar85}). 

\smallskip

The Euler forward discretization (with stepsize $\gamma>0$) of \eqref{CSGD} when $d=m$ and $\sigma=\sqrt{2}I_d$ is the following algorithm: 
\begin{equation}\label{LMC}
X_{k+1}=X_k-\gamma\nabla f(X_k)+\sqrt{2\gamma}\xi_k, \tag{LMC}
\end{equation}
where $\xi_k\sim\mathcal{N}(0,I_d)$ (multivariate standard normal distribution). This algorithm, which is known as Langevin Monte Carlo (see \cite{parisi}), is a standard sampling scheme, whose purpose is to generate samples from an approximation of a target distribution, in our case, proportional to $e^{-f(x)}$. Under appropriate assumptions on $f$, when $\gamma$ is small and $k$ is large such that $k\gamma$ is large, the distribution of $X_k$ converges in different topologies or is close in various metrics to the target distribution with density $\propto e^{-f(x)}$. Asymptotic and non-asymptotic (with convergence rates) results of this kind have been studied in a number of papers under various conditions; see \cite{user,guarantee,high,nona,und,hug} and references therein. By rescaling the problem, relation between sampling (\ie \eqref{LMC}) and optimization (\ie \eqref{sgdd}) has been also investigated for the strongly convex case in \eg \cite{user}.

\smallskip



Concerning \eqref{CSGD}, one can easily infer from \cite[Proposition~7.4]{Benaim99} that assuming $\sup_{x\in\R^d}\norm{\sigma(t,x)}_F=o(1/\sqrt{\log(t)})$, and conditioning on the event that $X(t)$ is bounded, we have almost surely that the set of limits of convergent sequences $X(t_k)$, $t_k \to +\infty$ is contained in $\argmin f$. Using results on asymptotic pseudo-trajectories from \cite{Benaim99}, the work of \cite{mertikopoulos_staudigl_2018,mdnetwork,acceleration} analyzed the behavior of the Stochastic Mirror Descent dynamics:
\begin{equation}\label{eq:smd}\tag{SMD}
\begin{aligned}
dY(t)&=-\nabla f(X(t))dt+\sigma(t,X) dW(t),\\
X(t)&=Q(\eta Y(t)),
\end{aligned}
\end{equation}
where $\calX \subset \R^d$ is a closed convex feasible region, $f$ is convex with Lipschitz continuous gradient on $\calX$, $Q:\R^d\rightarrow\calX$ is the mirror map induced by some strongly convex entropy, and $\eta>0$ is a sensitivity parameter. In \cite[Theorem~4.1]{mertikopoulos_staudigl_2018}, it is shown that if $\calX$ is also assumed bounded, that $\sup_{x\in\R^d}\norm{\sigma(t,x)}_F=o(1/\sqrt{\log(t)})$, and $Q$ satisfies some continuity assumptions\footnote{Compactness of $\calX$ and the condition on $\sigma(\cdot,\cdot)$ are clearly reminiscent of \cite[Proposition~7.4]{Benaim99}, though the latter is not discussed in \cite{mertikopoulos_staudigl_2018}.}, then the solution $X(t)$ \eqref{eq:smd} converges to a point in $\argmin f$ almost surely. Similar assumptions can be found in \cite{acceleration} to obtain almost sure convergence on the objective.
Let us observe that all these results do not apply to our setting. Indeed, if $\calX=\R^d$ (unconstrained problem), $Q(x)=x$ and $\eta=1$, we recover \eqref{CSGD}. Our work does not assume any boundedness whatsoever to establish our results. This comes however at somewhat stronger assumptions on $\sigma(\cdot,\cdot)$.

\smallskip

While finalizing this work, we became aware of the recent work of \cite{cooling}, which analyzes the behavior of \eqref{CSGD} for $f\in C^2(\R^d)$ not necessarily convex and which satisfies $\sup_{x\in\R^d}\norm{\sigma(\cdot,x)}_F\in \Lp^2(\R_+)$. Conditioning on the event that $\limsup_{t\rightarrow\infty}\norm{X(t)}<\infty$, they showed that $\nabla f(X(t))\rightarrow 0$ almost surely, almost sure convergence of $f(X(t))$, and if the objective $f$ is semialgebraic (and more generally tame), they also showed almost sure convergence of $X(t)$ towards a critical point of $f$. They also made attempt to get local convergence rates under the {\L}ojasiewicz inequality that are less transparent than ours. Our analysis on the other hand leverages convexity of $f$ to establish stronger results.

\subsection{Organization of the paper}

Section~\ref{sec:notation} introduces notations and reviews some necessary material from convex and stochastic analysis. Section~\ref{sec:smooth} states our main convergence results in the case of a convex differentiable objective function whose gradient is Lipschitz continuous. We first show the almost sure convergence of the process towards the set of minimizers, then we establish convergence rates for the values. Section~\ref{error_bound_Loja} introduces further geometric properties of the objective functions, namely {\L}ojasiewicz property and related error bound, which allows to obtain improved (local) convergence rates. This covers in particular the (locally) strongly convex case. In section~\ref{nonsmooth_structured}, we extend some results to the nonsmooth case by considering the additively structured "smooth + nonsmooth" convex minimization. We develop new stochastic differential equations that naturally lend themselves to splitting techniques. Technical lemmas and theorems that are needed throughout the paper are collected in the appendix.

%% file: tex/sec-Notation.tex
\section{Notation and Preliminaries}\label{sec:notation}

We will use the following shorthand notations:  given $d,n\in\N$,  $[n]\eqdef \{1,\ldots,n\}$, $\R^{d\times n}$ is the set of real matrices of size $d\times n$, and $I_d$ is the identity matrix of dimension $d$. For $M\in\R^{d\times n}$, $M^{\top}\in\R^{n\times d}$ is its transpose matrix and $\norm{M}_F$ is its Frobenius norm. For $M,M'\in\R^{d\times d}$, $M\preccurlyeq M'$ if and only if $u^{\top}(M'-M)u\geq 0$ for every $u\in\R^d$. For a set $\calD$, we denote its power set as $\calP(\calD)\eqdef \{\calC: \calC \subseteq \calD\}$. The sublevel of $f$ at height $r\in\R$ is denoted $[f\leq r]\eqdef \{x\in\R^d: f(x)\leq r\}$.




\subsection{On convex analysis}
Let us recall some important definitions and results from convex analysis in the finite-dimensional case; for a comprehensive coverage, we refer the reader to \cite{rocka}.

\smallskip

We denote by $\Gamma_0(\R^d)$ the class of proper lsc and convex functions on $\R^d$ taking values in $\Rinf$.
For $\mu > 0$, $\Gamma_{\mu}(\R^d) \subset \Gamma_0(\R^d)$ is the class of $\mu-$strongly convex functions. We denote by $C^{s}(\R^d)$ the class of $s$-times continuously differentiable functions on $\R^d$. 
For $L \geq 0$, $C_L^{1,1}(\R^d) \subset C^{1}(\R^d)$ is the set of functions on $\R^d$ whose gradient is $L-$Lipschitz continuous. 

The following \textit{Descent Lemma} which is satisfied by this class of functions plays a central role in optimization.
\begin{lemma}\label{descent}
Let $f\in C_L^{1,1}(\R^d)$, then $$f(y)\leq f(x)+\langle \nabla f(x), y-x\rangle+\frac{L}{2}
\norm{y-x}^2,\quad \forall x,y\in\R^d.$$
\end{lemma}
\begin{corollary}\label{2L}
Let $f\in C_L^{1,1}(\R^d)$ such that $\argmin f\neq\emptyset$, then 
\begin{equation*}
\norm{\nabla f(x)}^2\leq 2L(f(x)-\min f),\quad \forall x\in \R^d.\end{equation*}
\end{corollary}
\begin{proof}
Use Lemma \ref{descent} for an arbitrary $x\in\R^d$ and $y=x-\frac{1}{L}\nabla f(x)$. Then bound 
\[
f\pa{x-\frac{1}{L}\nabla f(x)}\geq \min f.
\]
\end{proof}

 The \textit{subdifferential} of a function $f\in\Gamma_0(\R^d)$ is the set-valued operator $\partial f:\R^d\rightarrow\calP(\R^d)$ such that, for every $x$ in $\R^d$,
\[
\partial f(x)=\{u\in\R^d:f(y)\geq f(x) + \dotp{u}{y-x} \qforallq y\in\R^d\}.
\]
When $f$ is continuous, $\partial f(x)$ is  non-empty convex and compact set for every $x\in \R^d$. If $f$ is differentiable, then $\partial f(x)=\{\nabla f(x)\}$. For every $x\in \R^d$ such that $\partial f(x) \neq \emptyset$, the minimum norm selection of $\partial f(x)$ is the unique element $\partial^0 f(x)\eqdef \argmin_{u\in \partial f(x)}\norm{u}$.

\subsection{On stochastic processes}
Let us recall some elements of stochastic analysis; for a more complete account, we refer to \cite{oksendal_2003,pardoux,mao}. Throughout the paper, $(\Omega,\calF,\PP)$ is a probability space and $\{\calF_t| t\geq 0\}$ is a filtration of the $\sigma-$algebra $\calF$. Given $\mathcal{C}\in\calP(\Omega)$, we will denote $\sigma(\mathcal{C})$ the $\sigma-$algebra generated by $\mathcal{C}$. We denote $\calF_{\infty}\eqdef \sigma\pa{\bigcup_{t\geq 0} \calF_t}\in\calF$.

The expectation of a random variable $\xi:\Omega\rightarrow\R^d$ is denoted by 
\[
\EE(\xi)\eqdef \int_{\Omega}\xi(\omega)d\PP(\omega).
\]
An event $E\in\calF$ happens almost surely if $\PP(E)=1$, and it will be denoted as "$E$, $\PP$-a.s." or simply "$E$, a.s.". The characteristic function of an event $E\in\calF$ is denoted by 
\[
\ind_E(\omega) \eqdef
\begin{cases}
1 & \text{if } \omega\in E,\\
0 & \text{otherwise}.
\end{cases}
\] 
An $\R^d$-valued stochastic process is a function $X:\Omega\times\R_+\rightarrow\R^d$. It is said to be continuous if $X(\omega,\cdot)\in C(\R_+;\R^d)$ for almost all $\omega\in\Omega$. We will denote $X(t)\eqdef X(\cdot,t)$. We are going to study \eqref{CSGD}, and in order to ensure the uniqueness of a solution, we introduce a relation over stochastic processes. Two stochastic processes $X,Y:\Omega\times [0,T]\rightarrow\R^d$ are said to be equivalent if $X(t)=Y(t)$, $\forall t\in [0,T]$, $\PP$-a.s. This leads us to define the equivalence relation $\calR$, which associates the equivalent stochastic processes in the same class. 

\smallskip

Furthermore, we will need some properties about the measurability of these processes. A stochastic process $X:\Omega\times\R_+\rightarrow\R^d$ is progressively measurable if for every $t\geq 0$, the map $\Omega\times[0,t]\rightarrow\R^d$ defined by $(\omega,s)\rightarrow X(\omega,s)$ is $\calF_t\otimes\calB([0,t])$-measurable, where $\otimes$ is the product $\sigma$-algebra and $\calB$ is the Borel $\sigma$-algebra. On the other hand, $X$ is $\calF_t$-adapted if $X(t)$ is $\calF_t$-measurable for every $t\geq 0$. It is a direct consequence of the definition that if $X$ is progressively measurable, then $X$ is $\calF_t$-adapted.

\smallskip

Let us define the quotient space:
\[
S_d^0[0,T] \eqdef \enscond{X:\Omega\times[0,T]\rightarrow\R^d}{X \text{ is a prog. measurable cont. stochastic process}}\Big/\calR.
\]
We set $S_d^0\eqdef \bigcap_{T\geq 0} S_d^0[0,T]$. Furthermore, for $\nu>0$, we define $S_d^{\nu}[0,T]$ as the subset of processes $X(t)$ in $S_d^0[0,T]$ such that 
\[
S_d^{\nu}[0,T]\eqdef \enscond{X\in S_d^0[0,T]}{\EE\pa{\sup_{t\in[0,T]}\norm{X_t}^{\nu}}<+\infty}.
\] 
We define $S_d^{\nu} \eqdef \bigcap_{T\geq 0} S_d^{\nu}[0,T]$.


  
Theorem~\ref{teoexistencia} in the appendix provides us with sufficient conditions to ensure the existence and uniqueness of the solution to \eqref{CSGD}. These conditions are met in our case under assumptions \eqref{H0} and \eqref{H}. 

Let us now present It\^o's formula which plays a central role in the theory of stochastic differential equations.
\begin{proposition}\label{itos}\cite[Chapter~4]{oksendal_2003}
Consider $X$ the solution of \eqref{CSGD}, $\phi: \R_+\times\R^d\rightarrow\R$ such that $\phi(\cdot,x)\in C^1(\R_+)$ for every $x\in\R^d$ and $\phi(t,\cdot)\in C^2(\R^d)$ for every $t\geq 0$. Then the process $$Y(t)=\phi(t,X(t)),$$ is an It\^o Process such that for all $t\geq 0$
\begin{multline}
Y(t)=Y(0)+\int_0^{t} \frac{\partial \phi}{\partial t}(s,X(s)) ds
-\int_0^{t} \dotp{\nabla \phi(s,X(s))}{\nabla f(X(s))} ds\\
+\int_0^{t}\dotp{\sigma^{\top}(s,X(s))\nabla \phi(s,X(s))}{dW(s)}+\frac{1}{2}\int_0^{t} \tr\pa{\sigma(s,X(s))\sigma^{\top}(s,X(s))\nabla^2\phi(s,X(s))}ds.
\end{multline}
Moreover, if for all $T>0$ 
\[
\EE\pa{\int_0^T \norm{\sigma^{\top}(s,X(s))\nabla \phi(s,X(s))}^2 ds}<+\infty,
\]
then $\displaystyle\int_0^{t}\dotp{\sigma^{\top}(s,X(s))\nabla \phi(s,X(s))}{dW(s)}$ is a square-integrable continuous martingale and
\begin{multline}
\EE[Y(t)]=Y(0)+\EE\pa{\int_0^{t} \frac{\partial \phi}{\partial t}(s,X(s)) ds}-\EE\pa{\int_0^{t} \dotp{\nabla \phi(s,X(s))}{\nabla f(X(s))} ds}\\
+\frac{1}{2}\EE\pa{\int_0^{t} \tr\pa{\sigma(s,X(s))\sigma^{\top}(s,X(s))\nabla^2\phi(s,X(s))}ds}.
\end{multline}

\end{proposition}
The $C^2$ assumption on $\phi(t,\cdot)$ in It\^o's formula is crucial. This can be weakened in certain cases leading to the following inequality that will be useful in our context.
\begin{proposition}\label{itoin}
Consider $X$ the solution of \eqref{CSGD}, $\phi_1\in C^1(\R_+)$, $\phi_2\in C_L^{1,1}(\R^d)$ and $\phi(t,x)=\phi_1(t)\phi_2(x)$. Then the process 
\[
Y(t)=\phi(t,X(t))=\phi_1(t)\phi_2(X(t)),
\] 
is an It\^o Process such that
\begin{multline}
Y(t) \leq Y(0)+\int_0^t \phi_1'(s)\phi_2(X(s)) ds - \int_0^{t} \phi_1(s)\dotp{\nabla \phi_2(X(s))}{\nabla f(X(s))} ds\\
+\int_0^{t}\dotp{\sigma^{\top}(s,X(s))\phi_1(s)\nabla \phi_2(X(s))}{dW(s)}
+\frac{L}{2}\int_0^{t} \phi_1(s)\tr\pa{\sigma(s,X(s))\sigma^{\top}(s,X(s))}ds.
\end{multline}

Moreover, if for all $T>0$ 
\[
\EE\pa{\int_0^T \norm{\sigma^{\top}(s,X(s))\phi_1(s)\nabla \phi_2(X(s))}^2 ds}<+\infty,
\]
then  
\begin{multline}
\EE[Y(t)]\leq Y(0)+\EE\pa{\int_0^t \phi_1'(s)\phi_2(X(s)) ds} - \EE\pa{\int_0^{t} \phi_1(s)\dotp{\nabla \phi_2(X(s))}{\nabla f(X(s))} ds}\\
+\frac{L}{2}\EE\pa{\int_0^{t} \phi_1(s)\tr\pa{\sigma(s,X(s))\sigma^{\top}(s,X(s))}ds}.
\end{multline}
\end{proposition}

\proof{Proof.}
Analogous to the proof of \cite[Proposition~C.2]{mertikopoulos_staudigl_2018}.
\endproof

%% file: tex/sec-Smooth.tex
\section{Convergence properties for convex differentiable functions}\label{sec:smooth}
We consider $f$ (called the potential) and study the dynamic \eqref{CSGD} under hypotheses \eqref{H0} (\ie $f \in C_L^{1,1}(\R^d) \cap \Gamma_0(\R^d)$) and \eqref{H}. Recall the definitions of $\sigma_*$ and $\sigma_{\infty}(t)$ from \eqref{eq:defsigstar}. Observe that from \eqref{H} one can take $\sigma_*^2=md\sup_{t \geq 0,x \in \R^d} |\sigma_{ik}(t,x)|^2$. Throughout the rest of the paper, we will use the shorthand notation
\[
\Sigma(t,x) \eqdef \sigma(t,x)\sigma(t,x)^{\top} . 
\]

\subsection{Almost sure convergence of trajectory}

Our first main result establish almost convergence of $X(t)$ to an $\calS$-valued random variable as $t \to +\infty$.

\begin{theorem}\label{converge2}
Consider the dynamic \eqref{CSGD} where $f$ and $\sigma$ satisfy the assumptions \eqref{H0} and \eqref{H}. Then, there exists a unique solution $X\in S_d^{\nu}$ of \eqref{CSGD}, for every $\nu\geq 2$.  Additionally, if $\sigma_{\infty}\in \Lp^2(\R_+)$, then: 
\begin{enumerate}[label=(\roman*)]
\item \label{acota} $\sup_{t\geq 0}\EE[\norm{ X(t)}^2]<+\infty$.
\vspace{1mm}
\item  $\forall x^{\star}\in \calS$, $\lim_{t\rightarrow\infty} \norm{ X(t)-x^{\star}}$ exists a.s. and $\sup_{t\geq 0}\norm{ X(t)}<+ \infty$ a.s.
\vspace{1mm}
\item \label{iiconv} 
$\lim_{t\rightarrow \infty}\norm{\nabla f(X(t))}=0$ a.s.  
As a result, $\lim_{t\rightarrow \infty} f(X(t))=\min f$ a.s.
\vspace{1mm}
\item  In addition to \ref{iiconv}, there exists an $\calS$-valued random variable $x^{\star}$ such that $\lim_{t\rightarrow\infty} X(t) = x^{\star}$ a.s.
 \end{enumerate}
\end{theorem}
\begin{proof}
The existence and uniqueness of a solution follows directly from the fact that the conditions of Theorem~\ref{teoexistencia} are satisfied under \eqref{H0} and \eqref{H}. The architecture of the proof of Theorem~\ref{converge2} consists of three steps that we briefly describe:
\begin{itemize}
\item The first step is based on It\^o's formula (Proposition~\ref{itos}).
Theorem \ref{impp} then allows us to conclude that for all $x^{\star}\in \calS$, $\lim_{t\rightarrow\infty} \norm{ X(t)-x^{\star}}$ exists a.s. Then, a separability argument is used to conclude that almost surely, for every $x^{\star}\in \calS$, $\lim_{t\rightarrow\infty} \norm{ X(t)-x^{\star}}$ exists. \\
\item  The second step consists in using another conclusion of Theorem \ref{impp} to conclude that $\norm{\nabla f(X(\cdot))}^2\in \Lp^1(\R_+)$ 
 a.s. After proving that this function is eventually uniformly continuous, we proceed according to Barbalat's Lemma  (see \cite{barbalat}) to conclude that $\lim_{t\rightarrow\infty}\norm{\nabla f(X(t))}=0$ a.s. As a consequence of the convexity of $f$ we deduce that $\lim_{t\rightarrow\infty}f(X(t))=\min f$ a.s.\\
\item Finally, the third step consists in using Opial's Lemma to conclude that there exists an $\calS$-valued random variable $x^{\star}$ such that $\lim_{t\rightarrow\infty} X(t)= x^{\star}$ a.s.
\end{itemize}

\begin{enumerate}[label=(\roman*)]
\item Let $x^{\star}$ be taken arbitrarily in $\calS$. Let us define the corresponding anchor function $\phi(x)=\frac{\norm{x-x^{\star}}^2}{2}$. Using It\^o's formula we obtain
\begin{align}
\phi(X(t))&=\underbrace{\frac{\norm{X_0-x^{\star}}^2}{2}}_{\xi}+\underbrace{\frac{1}{2}\int_0^t \tr\pa{\Sigma(s,X(s))}ds}_{A_t}-\underbrace{\int_0^t \dotp{\nabla f(X(s))}{X(s)-x^{\star}} ds}_{U_t} \nonumber \\
&+\underbrace{\int_0^t \dotp{\sigma^{\top}(s,X(s))\pa{X(s)-x^{\star}}}{dW(s)}}_{M_t}. 
\label{basic_Ito1}
\end{align}
Since $X\in S_d^2$ by Proposition~\ref{itos}, we have for every $T>0$, that 
\[
\EE\pa{\int_0^T\norm{\sigma^{\top}(s,X(s))\pa{X(s)-x^{\star}}}^2ds} \leq \EE\pa{\sup_{t\in [0,T]}\norm{ X(t)-x^{\star}}^2}\int_0^T \sigma_{\infty}^2(s)ds<+\infty.
\]
Therefore $M_t$ is a square-integrable continuous martingale. It is also  a continuous local martingale (see \cite[Theorem 1.3.3]{mao}), which implies that $\EE(M_t)=0$.

Let us now take the expectation of \eqref{basic_Ito1}. Using that 
\[
0\leq\tr\pa{\Sigma(s,X(s))} \leq \sigma_{\infty}^2(s) \qandq \langle \nabla f(X(s)),X(s)-x^{\star}\rangle\geq 0,
\] 
and  taking the supremum over $t\geq 0$, we obtain that
\[
\sup_{t\geq 0}\EE\pa{\frac{\norm{ X(t)-x^{\star}}^2}{2}}\leq \frac{\norm{ X_0-x^{\star}}^2}{2}+\frac{1}{2}\int_0^{\infty} \sigma_{\infty}^2(s)ds< +\infty.
\]
This shows the first claim.

\smallskip 
        
\item  $A_t$ and $U_t$ are two continuous adapted increasing processes with $A_0=U_0=0$ a.s.
Since $\phi(X(t))$ is nonnegative and $\sup_{x\in\R^d}\norm{ \sigma(\cdot,x)}_F\in \Lp^2(\R_+)$, we deduce that  $\lim_{t\rightarrow\infty}A_t< +\infty$. Then, we can use Theorem \ref{impp} to conclude that
\begin{equation}\label{ecconv}
\int_0^{\infty} \langle \nabla f(X(s)),X(s)-x^{\star}\rangle ds< +\infty \quad a.s.
\end{equation}
and
\begin{equation}\label{xkconv}
\forall x^{\star}\in \calS, \exists \Omega_{x^{\star}}\in\calF, \text{such that } \Pro(\Omega_{x^{\star}})=1 \text{ and } \lim_{t\rightarrow\infty}\norm{ X(\omega,t)-x^{\star}} \text{ exists } \forall \omega\in \Omega_{x^{\star}}.
\end{equation}
Since $\R^d$ is separable, there exists a countable set $Z\subseteq S$, such that $\cl(Z)=S$. Let $\tilde{\Omega}=\bigcap_{z\in Z}\Omega_z$. Since $Z$ is countable 
\[
\PP(\tilde{\Omega})=1-\PP\pa{\bigcup_{z\in Z}\Omega_z^c}\geq 1-\sum_{z\in Z}\PP(\Omega_z^c) = 1.
\]
For arbitrary $x^{\star}\in \calS$, there exists a sequence $(z_k)_{k\in\N}\subseteq Z$ such that $z_k\rightarrow x^{\star}$.
In view of \eqref{xkconv}, for every $k\in\N$ there exists $\tau_k:\Omega_{z_k}\rightarrow\R_+$ such that
\[ 
\lim_{t\rightarrow\infty}\norm{ X(\omega,t)-z_k}=\tau_k(\omega), \quad \forall\omega\in\Omega_{z_k}.
\]
Moreover, $\lim_{k\rightarrow\infty}\tau_k(\omega)$ exists since $(z_k)_{k\in\N}$ is convergent. 
Now, let $\omega \in \tilde{\Omega}$. Using the triangle inequality, we obtain that 
\[
\left|\norm{ X(\omega,t)-z_k}-\norm{ X(\omega,t)-x^{\star} }\right|\leq\norm{ z_k-x^{\star}}.
\]
Taking $\limsup_{t\rightarrow\infty}$ over the previous inequality, we conclude that 
\[
\abs{\tau_k(\omega)-\limsup_{t\rightarrow\infty}\norm{ X(\omega,t)-x^{\star}}}\leq\norm{ z_k-x^{\star}}.
\]
A similar conclusion holds for the $\liminf_{t\rightarrow\infty}$. Then, taking the limit over $k$, we deduce
\[
\lim_{t\rightarrow\infty}\norm{ X(\omega,t)-x^{\star} }=\lim_{k\rightarrow\infty}\tau_{k}(\omega),\quad\forall \omega\in\tilde{\Omega} ,
\]
whence we obtain that the previous limit exists on a set of probability $1$ independently of $x^{\star}$.
    
\smallskip
    
Let us recall that there exists $\Omega_c\in\calF$  such that $\Pro(\Omega_c)=1$ and $X(\omega,\cdot)$ is continuous for every $\omega\in\Omega_c$. Now let $x^{\star}\in \calS$ arbitrary, since the limit exists, for every $\omega\in\tilde{\Omega}\cap\Omega_c$ there exists $T(\omega)$ such that $\norm{ X(\omega,t)-x^{\star}}\leq 1$ for every $t\geq T(\omega)$. Besides, since $X(\omega,\cdot)$ is continuous, by Bolzano's theorem $\sup_{t\in [0,T(\omega)]}\norm{ X(\omega,t)}=\max_{t\in [0,T(\omega)]}\norm{ X(\omega,t)} \eqdef h(\omega) <+\infty$. Therefore,  $\sup_{t\geq 0}\norm{ X(t)}<\max\{h(\omega),1+\norm{ x^{\star}}\}<\infty$.
    
\smallskip
    
\item By convexity of $f$ and \eqref{ecconv}, we have that there exists $\Omega_f\in\calF$ such that $\Pro(\Omega_f)=1$ and $f(X(\omega,\cdot))-\min f\in \Lp^1(\R_+)$ for every $\omega\in\Omega_f$. By Corollary~\ref{2L}, we obtain that $\norm{\nabla f(X(\omega,\cdot))}\in\Lp^2(\R_+)$ for every $\omega\in\Omega_f$. Let $\omega\in\Omega_f$ arbitrary, then $\liminf_{t\rightarrow\infty} \norm{ \nabla f(X(\omega,t))}=0$. If $\limsup_{t\rightarrow\infty} \norm{ \nabla f(X(\omega,t))}=0$ then we conclude. Suppose by contradiction that $\limsup_{t\rightarrow\infty} \norm{ f(X(\omega,t))}>0$. Then, by Lemma~\ref{existenceof}, there exists $\delta>0$ satisfying 
\[
0=\liminf_{t\rightarrow\infty} \norm{\nabla f(X(\omega,t))}<\delta<\limsup_{t\rightarrow\infty} \norm{\nabla f(X(\omega,t))},
\] 
and there exists $(t_k)_{k\in\N}\subset \R_+$ such that $\lim_{k\rightarrow\infty} t_k=\infty$, 
\[ 
\norm{\nabla f(X(\omega,t_k))}>\delta \qandq t_{k+1}-t_k>1, \quad \forall k\in\N.
\]
Let $M_t=\displaystyle\int_0^t \sigma(s,X(s))dW(s)$. This is a continuous martingale (w.r.t. the filtration $\calF_t$), which verifies
\[
\EE(|M_t|^2)=\EE\pa{\int_0^t \norm{\sigma(s,X(s))}_F^2ds}\leq \EE\pa{\int_0^{\infty} \sigma_{\infty}^2(s)ds}<\infty, \forall t\geq 0.
\]
According to Theorem \ref{doob}, we deduce that there exists a random variable $M_{\infty}$ w.r.t. $\calF_{\infty}$, and which verifies: $\EE(|M_{\infty}|^2)<+\infty$, and there exists $\Omega_M\in\calF$ such that $\Pro(\Omega_M)=1$ and 
\[
\lim_{t\rightarrow\infty}M_t(\omega)= M_{\infty}(\omega) \mbox{ for every } \omega\in\Omega_M .
\]
Let $\Omega_{\mathrm{conv}} \eqdef \tilde{\Omega}\cap\Omega_c\cap\Omega_f\cap\Omega_M$, hence $\Pro(\Omega_{\mathrm{conv}})=1$. Take any $\omega_0\in\Omega_{\mathrm{conv}}$. We allow ourselves the abuse of notation $X(t)\eqdef X(\omega_0,t)$ during the rest of the proof from this point.
 
\smallskip

Let $\varepsilon\in \left]0,\min\pa{\frac{\delta^2}{4L^2},L}\right[$. Note that $([t_k,t_k+\frac{\varepsilon}{2L}]:k\in\N)$ are disjoint intervals. On the other hand, according to the convergence property of $M_t$ and the fact that $\norm{\nabla f(X(\cdot))}\in \Lp^2(\R_+)$, there exists $k'>0$ such that for every $k\geq k'$
\[
\sup_{t\geq t_k}|M_t-M_{t_k}|^2<\frac{\varepsilon}{4} \qandq \int_{t_k}^{\infty}\norm{\nabla f (X(s))}^2ds\leq \frac{L}{2}.
\]
Besides, for every $k\geq k'$, $t\in [t_k,t_k+\frac{\varepsilon}{2L}]$
\[
\norm{ X(t)-X(t_k)}^2\leq 2(t-t_k)\int_{t_k}^t\norm{\nabla f(X(s))}^2 ds+2|M_t-M_{t_k}|^2\leq 2(t-t_k)\frac{L}{2}+\frac{\varepsilon}{2}\leq \varepsilon.
\]
Since $f\in C_L^{1,1}(\R^d)$, we have that for every $k\geq k'$ and $t\in [t_k,t_k+\frac{\varepsilon}{2L}]$
\[
\norm{ \nabla f(X(t))-\nabla f(X(t_k))}^2\leq L^2\norm{ X(t)-X(t_k)}^2\leq \pa{\frac{\delta}{2}}^2.
\]
Therefore, for every $k\geq k'$, $t\in [t_k,t_k+\frac{\varepsilon}{2L}]$
\[
\norm{\nabla f(X(t))}\geq  \norm{ \nabla f(X(t_k))}-\underbrace{\norm{ \nabla f(X(t))- \nabla f(X(t_k))}}_{\leq \frac{\delta}{2}}\geq \frac{\delta}{2}.
\]
Finally, 
\[
\int_0^{\infty} \norm{ \nabla f(X(s))}^2 ds\geq \sum_{k\geq k'}\int_{t_k}^{t_k+\frac{\varepsilon}{2L}} \norm{ \nabla f(X(s))}^2 ds\geq\sum_{k\geq k'}\frac{\delta^2\varepsilon}{8L}=\infty ,
\]
which contradicts $\norm{\nabla f(X(\cdot))}\in \Lp^2(\R_+)$. So, 
\[
\limsup_{t\rightarrow\infty} \norm{\nabla f(X(\omega,t))}=\liminf_{t\rightarrow\infty} \norm{\nabla f(X(\omega, t))} =\lim_{t\rightarrow\infty} \norm{\nabla f(X(\omega,t))}=0,\quad \forall\omega\in \Omega_{\mathrm{conv}}.
\]
Let $x^{\star}\in \calS$ and $\omega\in\Omega_{\mathrm{conv}}$ taken arbitrary. By convexity and Cauchy-Schwarz inequality:
\[
0\leq f(X(\omega,t))-\min f\leq \norm{\nabla f(X(\omega,t))} \norm{ X(\omega,t)-x^{\star}}.
\]
The claim then follows as we have already obtained that $\lim_{t\rightarrow 0}\norm{ X(\omega,t)-x^{\star}}$ exists, and 
\[\lim_{t\rightarrow \infty}\norm{\nabla f(X(\omega,t))}=0.\]
 
\item Let $\omega\in\Omega_{\mathrm{conv}}$ and $\bar{x}(\omega)$ be a sequential limit point of $X(\omega,t)$. Equivalently, there exists an increasing sequence $ (t_k)_{k\in\N}\subset \R_+$ such that $\lim_{k\rightarrow\infty} t_k=\infty$ and 
\[
\lim_{k\rightarrow\infty} X(\omega, t_k) = \bar{x}(\omega).
\]
Since $\lim_{t\rightarrow\infty} f(X(\omega,t))=\min f$ and by continuity of $f$, we obtain directly that $\bar{x}(\omega)\in \calS$. Finally by Opial's Lemma (see \cite{opial}) we conclude that there exists $x^{\star}(\omega)\in \calS$ such that $\lim_{t\rightarrow\infty}X(\omega,t)= x^{\star}(\omega)$. In other words, since $\omega\in\Omega_{\mathrm{conv}}$ was arbitrary, there exists an $\calS$-valued random variable $x^{\star}$ such that $\lim_{t\rightarrow\infty} X(t)= x^{\star}$ a.s.
\end{enumerate}
\end{proof}


\subsection{Convergence rates of the objective}
Our first result, stated below, summarizes the global convergence rates in expectation satisfied by the trajectories of \eqref{CSGD}.

\begin{theorem}\label{importante0}
Consider the dynamic \eqref{CSGD} where $f$ and $\sigma$ satisfy the assumptions \eqref{H0} and \eqref{H}. The following statements are satisfied by the solution trajectory $X \in S_d^2$ of \eqref{CSGD}:

\smallskip

\begin{enumerate}[label=(\roman*)]
\item \label{0i} Let $\displaystyle\overline{f\circ X}(t)\eqdef t^{-1}\int_0^t f(X(s))ds$ and $\displaystyle\overline{X}(t)=t^{-1}\int_0^t X(s)ds$. Then 
\begin{equation}\label{eq:0i}
\EE\pa{f(\overline{X}(t))-\min f}\leq \EE\pa{\overline{f\circ X}(t)-\min f}\leq \frac{\dist(X_0,\calS)^2}{2t}+\frac{\sigma_*^2}{2}, \quad \forall t> 0.
\end{equation}
Besides, if $\sigma_{\infty}$ is $\Lp^2(\R_+)$, then 
\begin{equation}
\EE\pa{f(\overline{X}(t))-\min f}\leq \EE\pa{\overline{f\circ X}(t)-\min f}=\calO\pa{\frac{1}{t}} .
\end{equation}

\item \label{0ii} If moreover $f\in \Gamma_{\mu}(\R^d)$ with $\mu > 0$, then $\calS=\{x^{\star}\}$ and 
\begin{equation}
\EE\pa{\norm{ X(t)-x^{\star}}^2}\leq \norm{X_0-x^{\star}}^2e^{-2\mu t}+\frac{\sigma_*^2}{2\mu}, \quad \forall t\geq 0.
\end{equation}
Besides, if $\sigma_{\infty}$ is decreasing and vanishes at infinity, then for every $ \lambda\in ]0,1[$: 
\begin{equation}\label{conv_rate_strongly_convex}
\EE\pa{\norm{X(t)-x^{\star}}^2}\leq \norm{X_0-x^{\star}}^2e^{-2\mu t}+\frac{\sigma_*^2}{2\mu}e^{-2\mu (1-\lambda)t}+\sigma_{\infty}^2(\lambda t), \quad \forall t\geq 0.
\end{equation}
\end{enumerate}
\end{theorem}

\begin{proof}
\begin{enumerate}[label=(\roman*)]
\item Let $x^{\star} \in \calS$. Let $g(t)=\phi(X(t))=\frac{\norm{X(t)-x^{\star}}^2}{2}$ and $G(t)=\EE(g(t))$. By applying Proposition~\ref{itos} with $\phi$, and using the convexity of $f$, we obtain 
\begin{align}
G(t)-G(0)&=\EE\pa{\int_0^t \langle \nabla f(X(s)), x^{\star}-X(s) \rangle ds}+\frac{1}{2}\EE\pa{\int_0^t  \tr[\Sigma(s,X(s))]ds} \nonumber \\
&\leq -\EE\pa{\int_0^t (f(X(s))-\min f) ds}+\frac{1}{2}\EE\pa{\int_0^t  \tr[\Sigma(s,X(s))]  ds} \label{convex}\\
&\leq -\EE\pa{\int_0^t (f(X(s))-\min f) ds}+\frac{\sigma_*^2 }{2}t \nonumber.
\end{align}
Then rearranging the terms in \eqref{convex}, using $G(t)\geq 0$, and dividing by $t>0$, we obtain
\begin{equation}\label{conv}
\frac{1}{t}\EE\pa{\int_0^t (f(X(s))-\min f) ds}\leq \frac{\norm{ X_0-x^{\star}}^2}{2t}+\frac{\sigma_*^2}{2},\quad\forall t>0 .
\end{equation}
Since $x^{\star}$ is arbitrary, by taking the infimum with respect to $x^{\star} \in \calS$ in \eqref{conv}, we obtain
\begin{equation}\label{conv1}
\frac{1}{t}\EE\pa{\int_0^t (f(X(s))-\min f) ds}\leq \frac{\dist(X_0,\calS)^2}{2t}+\frac{\sigma_*^2}{2},\quad\forall t>0 .
\end{equation}
Moreover, if $\sigma_{\infty}\in \Lp^2(\R_+)$, then using  inequality \eqref{convex}, we have 
\[
G(t)-G(0)\leq-\EE\pa{\int_0^t (f(X(s))-\min f) ds}+\frac{1}{2}\pa{\int_0^{+\infty}  \sigma_{\infty}^2
(s)ds}.
\]
Rearranging as before, we conclude that 
\begin{equation}\label{conv0}
\frac{1}{t}\EE\pa{\int_0^t (f(X(s))-\min f) ds}\leq \frac{\dist(X_0,\calS)^2}{2t}+\frac{1}{2t}\int_0^{+\infty}\sigma_{\infty}^2(s) ds, \quad \forall t>0 .
\end{equation}
Then complete the result with the inequality
\[
\EE\pa{f(\overline{X}(t))-\min f}\leq \EE\pa{\overline{f\circ X}(t)-\min f}
\]
which follows from convexity of $f$ and Jensen's inequality.

\smallskip 

\item 
    
Let $g(t)=\phi(X(t))=\frac{\norm{ X(t)-x^{\star}}^2}{2}, G(t)=\EE(g(t))$. By Proposition \ref{itos} with $\phi$, we obtain
\begin{equation} \label{mia}
G(t)-G(0)=\EE\pa{\int_0^t\langle -\nabla f(X(s)),X(s)-x^{\star}\rangle ds}+\frac{1}{2}\EE\pa{\int_0^t \tr[\Sigma(s,X(s))]ds}.
\end{equation}
Using that $f\in\Gamma_{\mu}(\R^d)$, we deduce that
\begin{align*}
G(t)&\leq G(0) -2\mu \int_{0}^t G(t)+\int_0^t\frac{\sigma_*^2}{2}, \quad\forall t\geq 0.
\end{align*}
In order to invoke Lemma~\ref{comparison}, we solve the ODE 
\[
\begin{cases}
y'(t)&=-2\mu y(t)+\frac{\sigma_*^2}{2},\quad t>0\\
y(0)&=\frac{\norm{ X_0-x^{\star}}^2}{2}.
\end{cases}
\]
Solving it by the integrating factor method, we conclude that
\[
G(t)\leq\frac{\norm{ X_0-x^{\star}}^2}{2}e^{-2\mu t}+\frac{\sigma_*^2}{4\mu} ,\quad \forall t\geq 0.
\]
Combining this inequality with the gradient descent Lemma \ref{descent}, we obtain 
\begin{equation}\label{scr2}
\EE[f(X(t))-\min f]\leq L\pa{\frac{\norm{ X_0-x^{\star}}^2}{2}e^{-2\mu t}+\frac{\sigma_*^2}{4\mu}},\quad \forall t\geq 0.
\end{equation}
Suppose now that $\sigma_{\infty}$ is decreasing and vanishes at infinity. We can bound the trace term by $\sigma_{\infty}^2$ in \eqref{mia}. To use Lemma~\ref{comparison}, we need to solve 
\[
\begin{cases}
y'(t)&= -2\mu y(t)+\frac{\sigma_{\infty}^2(t)}{2},\quad t>0 \\
y(0)&= \frac{\norm{X_0-x^{\star}}^2}{2}.
\end{cases}
\]
Let $\lambda\in ]0,1[$, using the integrating factor method, we get 
\begin{align*}
y(t)&\leq y(0)e^{-2\mu t}+e^{-2\mu t}\int_0^t\frac{\sigma_{\infty}^2(s)}{2}e^{2\mu s}ds\\
&\leq y(0)e^{-2\mu t}+e^{-2\mu t}\pa{\int_0^{\lambda t}\frac{\sigma_{\infty}^2(s)}{2}e^{2\mu s}ds+\int_{\lambda t}^{t}\frac{\sigma_{\infty}^2(s)}{2}e^{2\mu s}ds}\\
&\leq y(0)e^{-2\mu t}+e^{-2\mu t}\pa{\frac{\sigma_*^2}{2}\int_0^{\lambda t}e^{2\mu s}ds+\frac{\sigma_{\infty}^2(\lambda t)}{2}\int_{\lambda t}^{t}e^{2\mu s}ds}\\
&\leq y(0)e^{-2\mu t}+e^{-2\mu t}\pa{\frac{\sigma_*^2}{4\mu}e^{2\mu \lambda t}+\frac{\sigma_{\infty}^2(\lambda t)}{2}e^{2\mu t}}, \hspace{0.2 cm} \forall t\geq 0.
\end{align*}
According to Lemma \ref{comparison}, we deduce that
\[
G(t)\leq \frac{\norm{X_0-x^{\star}}^2}{2}e^{-2\mu t}+\frac{\sigma_*^2}{4\mu}e^{-2\mu( 1-\lambda) t}+\frac{\sigma_{\infty}^2(\lambda t)}{2}, \hspace{0.2 cm} \forall t\geq 0,
\]
which is our claim \eqref{conv_rate_strongly_convex}.
\end{enumerate}
\end{proof}

Under a stronger assumption on $\sigma_{\infty}$, we also have the following pointwise sublinear convergence rate in expectation.
\begin{proposition}\label{beta}
Consider the dynamic \eqref{CSGD} where $f$ and $\sigma$ satisfy the assumptions \eqref{H0} and \eqref{H}. Assume that there exists $K\geq 0, \beta\in [0,1[$ such that
\begin{equation}\label{eq:assumbeta}
\int_0^t (s+1)\sigma_{\infty}^2(s)ds\leq Kt^{\beta}, \quad \forall t\geq 0.
\end{equation}
Then the solution trajectory $X \in S_d^2$ of \eqref{CSGD} satisfies
\[
\EE\pa{f(X(t))-\min f} = \calO(t^{\beta-1}).
\]
\end{proposition}
\begin{proof}
Given $x^{\star}\in \calS$, let us apply Proposition \ref{itoin} successively with  $V_1(t,x)=t(f(x)-\min f)$, then with $V_2(x)=\frac{1}{2}\norm{x-x^{\star}}^2$. Taking the expectation and adding the two results, we get
\begin{align*}
\EE\pa{V_1(t,X(t))+V_2(X(t))}& \leq \frac{1}{2}\norm{X_0-x^{\star}}^2+\frac{L}{2}\int_0^t s\sigma_{\infty}^2(s)ds+\frac{1}{2}\int_0^t \sigma_{\infty}^2(s)ds\\
&\leq \frac{1}{2}\norm{X_0-x^{\star}}^2+\frac{\max\{1,L\}}{2}\pa{\int_0^t (s+1)\sigma_{\infty}^2(s)ds},
\end{align*}
where we have used the convexity of $f$ in the first inequality. Then we conclude that 
\[
\EE(f(X(t))-\min f)\leq\frac{\norm{X_0-x^{\star}}^2}{2t}+\frac{K\max\{1,L\}}{2}t^{\beta-1}=\calO(t^{\beta-1}).
\]
\end{proof}

When $f$ is also $C^2$, we get an improved $o(t^{-1})$ global convergence rate on the objective in almost sure sense.

\begin{theorem}
Consider the dynamic \eqref{CSGD}. Assume that $f \in C^2(\R^d))$ such that $\Hess f\preccurlyeq LI_d$ and satisfies assumption \eqref{H0}, and that $\sigma$ satisfies assumption \eqref{H} and that  $t\mapsto t\sigma_{\infty}^2(t)\in \Lp^1(\R_+)$. Then, the solution trajectory $X \in S_d^2$  of \eqref{CSGD} obeys:
\begin{enumerate}[label=(\roman*)]
    \item $t\mapsto t\Vert \nabla f(X(t))\Vert^2\in \Lp^1(\R_+)$ a.s.
    \item $f(X(t))-\min f=o(t^{-1})$ a.s.
\end{enumerate}
\end{theorem}

\begin{proof}
By applying It\^o's formula in Proposition~\ref{itos} with $\phi(t,x)=t(f(x)-\min f)$ we get 
\begin{align*}
t(f(X(t))-\min f)&=\int_{0}^t f(X(s))-\min f ds +\frac{1}{2} \int_0^t s\tr[\Sigma(s,X(s))\Hess f(X(s))]ds\\
&-\int_0^t s\Vert\nabla f(X(s))\Vert^2ds+\int_0^t\langle s\sigma^{\top}(s,X(s))\nabla f(X(s)),dW(s)\rangle.
\end{align*}
By \eqref{ecconv} and convexity of $f$, we deduce that $f(X(\cdot))-\min f\in \Lp^1(\R_+)$ a.s. Moreover, 
\[
\int_{0}^{\infty} s\tr[\Sigma(s,X(s))\Hess f(X(s))]ds\leq L\int_0^{\infty} s\sigma_{\infty}^2(s)ds<+\infty.
\]
Then by Theorem \ref{impp}, we have that $\lim_{t\rightarrow\infty} t(f(X(t))-\min f)$ exists a.s. and $\int_0^{\infty} t\Vert\nabla f(X(t))\Vert^2dt<+\infty$ a.s. Finally, by Lemma~\ref{lim0}, we conclude that $\lim_{t\rightarrow\infty} t(f(X(t))-\min f)=0$ a.s. 
\end{proof}
\smallskip

\section{Convergence rates under {\L}ojasiewicz inequality}\label{error_bound_Loja}
The local convergence rate of the first-order descent methods can be understood using the {\L}ojasiewicz  property and the associated {\L}ojasiewicz exponent, see \cite{Attouch2013convergence,Frankel}. The {\L}ojasiewicz property has its roots in algebraic geometry, and it essentially describes a relationship between the objective value and its gradient (or subgradient). 
\begin{definition}[{\L}ojasiewicz inequality]\label{def:lojq}
Let $f:\R^d\rightarrow\R$ be a differentiable with $\calS=\argmin (f)\neq\emptyset$ and $q\in [0,1[$. $f$ satisfies the {\L}ojasiewicz inequality with exponent $q$ at $\bar{x} \in \calS$ if there exists a neighborhood $\calV_{\bar{x}}$ of $\bar{x}$, $r>\min f$ and $\mu > 0$ such that
\begin{equation}\label{li}
\mu(f(x)-\min f)^{q}\leq \norm{\nabla  f(x)},\quad \forall x \in \calV_{\bar{x}} \cap [\min f < f < r] .
\end{equation}
The function $f$ has the {\L}ojasiewicz property on $\calS$ if it obeys \eqref{li} at each point of $\calS$ with the same constant $\mu$ and exponent $q$, and we will write $f \in \Loj^q(\calS)$. 
\end{definition}

Error bounds have also been successfully applied to various branches of optimization, and in particular to complexity analysis, see \cite{Pang97}. Of particular interest in our setting is the H\"olderian error bound.
\begin{definition}[H\"olderian error bound]
Let $f:\R^d\rightarrow\R$ be a proper function such that $\calS=\argmin (f) \neq \emptyset$. $f$ satisfies a H\"olderian (or power-type) error bound inequality with exponent $p \geq 1$, and we write $f \in \EB^p$, if there exists $\gamma>0$ and $r>\min f$ such that
\begin{equation}\label{eq:errbnd}
f(x)-\min f \geq \gamma\dist(x,\calS)^p, \quad \forall x \in [\min f \leq f \leq r] .
\end{equation}
For a given $r>\min f$ such that \eqref{eq:errbnd} holds, we will use the shorhand notation $f\in \EB^p([f \leq r])$.
\end{definition}

A deep result due {\L}ojasiewicz states that for arbitrary continuous semi-algebraic functions, the H\"olderian error bound inequality holds on any compact set, and the {\L}ojasiewicz inequality holds at each point; see \cite{loj1,loj2}. In fact, for convex functions, the {\L}ojasiewicz  property and H\"olderian error bound are actually equivalent. 

\begin{proposition}\label{ebw} 
Assume that $f\in\Gamma_0(\R^d) \cap C^{1}(\R^d)$ with $\calS=\argmin (f) \neq \emptyset$. Let $q\in[0,1[$, $p\eqdef \frac{1}{1-q}\geq 1$ and $r > \min f$. Then $f$ verifies the {\L}ojasiewicz inequality \eqref{li} at $\bar{x} \in \calS$ if and only if the H\"olderian error bound \eqref{eq:errbnd} holds on $\calV_{\bar{x}} \cap [\min f < f < r]$.
\end{proposition}

\begin{proof}
Combine \cite[Lemma~4 and Theorem~5]{BolteKLComplexity16}. 
\end{proof}

\medskip

We are now ready to state the following ergodic local convergence rate.
\begin{proposition}\label{corimportante0} Consider the hypotheses of Theorem~\ref{importante0} and let $\varepsilon>0$. If $f\in \EB^p([f\leq r_{\varepsilon}])$ for $r_{\varepsilon}>\min f+\frac{\sigma_*^2}{2}+\varepsilon$, then $\exists t_{\varepsilon}>0$ such that
\[
\dist\pa{\EE(\overline{X}(t)),S}=\calO(t^{-\frac{1}{p}})+\calO\pa{\sigma_*^{\frac{2}{p}}}, \quad\forall t\geq t_{\varepsilon}.
\]
\end{proposition}

\begin{proof}
There exists $t_{\varepsilon}>0$ such that for all $t\geq t_{\varepsilon}$, $\frac{\dist(X_0,\cal\calS)^2}{2t}<\varepsilon$. Thus, from \eqref{eq:0i} and Jensen's Inequality, we have 
\[
f\pa{\EE[\overline{X}(t)]} \leq \EE[f(\overline{X}(t))]\leq \min f+\frac{\sigma_*^2}{2}+\varepsilon\leq r_{\varepsilon},\quad\forall t\geq t_{\varepsilon}.
\]
Clearly, $\EE[\overline{X}(t)]\in [f\leq r_{\varepsilon}]$ for $t\geq t_{\varepsilon}$. Using Theorem~\ref{importante0} and that $f\in \EB^p([f\leq r_{\varepsilon}])$, letting $\gamma>0$ the coefficient of the error bound, we have 
\[
\gamma\dist(\EE(\overline{X}(t)),\calS)^p\leq f(\EE[\overline{X}(t)])-\min f
\leq \frac{\dist(X_0,\calS)^2}{2t}+\frac{\sigma_*^2}{2},\quad\forall t\geq t_{\varepsilon}.
\]
Since $p \geq 1$, Jensen's inequality yields
\[
\dist(\EE(\overline{X}(t)),\calS)\leq \pa{\frac{\dist(X_0,\calS)^2}{2\gamma_r }}^{\frac{1}{p}}t^{-\frac{1}{p}}+\pa{\frac{\sigma_*^2}{2\gamma_r}}^{\frac{1}{p}},\quad\forall t\geq t_{\varepsilon}.
\]
\end{proof}

\subsection{Discussion on the localization of the process}\label{subsec:locprocess}
Let us take a moment to elaborate on the localization of the process $X(t)$ generated by \eqref{CSGD} when $f\in C_L^{1,1}(\R^d)\cap\Gamma_0(\R^d)$ and $\sigma_{\infty}\in\Lp^2(\R_+)$. This discussion is essential to understand the challenges underlying the analysis of the local convergence properties and rates in a stochastic setting under (local) error bounds. First, observe that the hypothesis of Lipschitz continuity of the gradient is incompatible with a global hypothesis of error bound or {\L}ojasiewicz inequality unless the exponent is $p=2$ or $q=\frac{1}{2}$, respectively. Therefore, we can only ask for these inequalities to be locally satisfied. Even though, thanks to convexity, we could introduce a global desingularizing function (see \cite[Theorem~3]{BolteKLComplexity16}), this function would not be concave nor convex, a fundamental property usually at the heart of the local analysis. In recent literature on stochastic processes and local properties, it is usual to find hypotheses about the almost sure localization of the process or that it is essentially bounded. Nevertheless, these assumptions are unrealistic or outright false due to the behavior of the Brownian Motion. Hence, we will avoid making these kinds of assumptions.

\smallskip

What we will do is to consider that by Theorem~\ref{converge2} we have that $\lim_{t\rightarrow\infty}f(X(t))=\min f$ a.s., which means that there exists $\Omega_{\mathrm{conv}}\in\calF$ such that $\Pro(\Omega_{\mathrm{conv}})=1$, and $(\forall r>\min f,\forall \omega\in\Omega_{\mathrm{conv}})$, $(\exists t_{r}(\omega)>0)$ such that $(\forall t>t_{r}(\omega))$, $X(\omega,t)\in [f\leq r]$. However, one should not infer from this that $X(t)\in [f\leq r]$ a.s. for $t$ large enough. Indeed, $t_{r}$ is a random variable which cannot be in general bounded uniformly on $\Omega_{\mathrm{conv}}$. Unfortunately, this flawed argument appears quite regularly in the literature. Rather, in this paper, we will invoke measure theoretic arguments to pass from a.s. convergence to almost uniform convergence thanks to Egorov's theorem (see Theorem~\ref{egorov}). More precisely, we will show that 
\begin{tcolorbox}
\begin{center}$(\forall \delta>0,\forall r>\min f), (\exists \Omega_{\delta}\in \calF \text{ s.t. } \Pro(\Omega_{\delta})\geq 1-\delta \text{ and } \exists \hat{t}_{r,\delta}>0), (\forall \omega\in\Omega_{\delta},\forall t>\hat{t}_{r,\delta}),$\\
$X(\omega,t)\in [f\leq r].$\end{center}
\end{tcolorbox}

Hence, this property will allow us to localize $X(t)$ in the sublevel set of $f$ at $r$ for $t$ large enough with probability at least $1-\delta$. In turn, we will be able to invoke the error bound (or {\L}ojasiewicz) inequality.

\subsection{Convergence rates under {\L}ojasiewicz Inequality}
Let $\sigma_{\infty}\in\Lp^2(\R_+),L>0, \delta>0,\beta\in[0,1[$ and some positive constants $C_*, C_{**}, C_K$. Consider the functions $h_{\delta},l_{\delta},k_{\delta}:\R_+\rightarrow\R$ defined by:
\begin{align}
h_{\delta}(t)&=\sigma_{\infty}^2(t)+C_*\sqrt{\delta}  \frac{\sigma_{\infty}^2(t)}{2\sqrt{\int_{\hat{t}_{\delta}}^t \sigma_{\infty}^2(u)du}}, \label{eq:hdelta}\\
l_{\delta}(t)&=\frac{L}{2}\sigma_{\infty}^2(t)+C_{**}\sqrt{\delta}\frac{\sigma_{\infty}^2(t)}{2\sqrt{\int_{\hat{t}_{\delta}}^t \sigma_{\infty}^2(u)du}}, \label{eq:ldelta}\\
k_{\delta}(t)&=\frac{L}{2}\sigma_{\infty}^2(t)+C_{K}\sqrt{\delta}  \frac{\sigma_{\infty}^2(t)t^{\beta-1}}{2\sqrt{\int_{\hat{t}_{\delta}}^t \sigma_{\infty}^2(u)u^{\beta-1}du}}. \label{eq:kdelta}
\end{align}

We are now ready to state our main local convergence result.

\begin{theorem}\label{importante3}
Consider $X\in S_d^2$ the solution trajectory of \eqref{CSGD} where $f$ and $\sigma$ satisfy the assumptions \eqref{H0} and \eqref{H}, and suppose that $\sigma_{\infty}\in \Lp^2(\R_+)$ ($C_{\infty}\eqdef \norm{\sigma_{\infty}}_{\Lp^2(\R_+)}$). Let $p\geq 2$ and $q\eqdef 1-\frac{1}{p}\in [\frac{1}{2},1[$, and assume that $f \in \Loj^q(\calS)$. Consider also the positive constants $C_*,C_{**},C_K,C_d,C_f$ (detailed in the proof). 
Then, for all $\delta>0$, there exists a measurable set $\Omega_{\delta}$ such that $\PP(\Omega_{\delta}) \geq 1-\delta$ and $\hat{t}_{\delta}>0$ such that the following statements hold.
\begin{enumerate}[label=(\roman*)]
\item \label{p2} If $p=2$ and $\sigma_{\infty}$ is decreasing, then $\sigma_{\infty}$ vanishes at infinity and
\begin{enumerate}[label=(\alph*)]
\item \label{p2a}
there exists $\gamma>0$ such that for every $\lambda \in ]0,1[$, 
\begin{equation}
\begin{aligned}
\EE\pa{\frac{\dist(X(t),\calS)^2}{2}}&\leq e^{-2\gamma(t-\hat{t}_{\delta})}\EE\pa{\frac{\dist(X(\hat{t}_{\delta}),\calS)^2}{2}}\\
&+e^{-2\gamma (1-\lambda)(t-\hat{t}_{\delta})}(C_{\infty}^2+C_*C_{\infty}\sqrt{\delta})\\
&+\frac{h_{\delta}(\hat{t}_{\delta}+\lambda(t-\hat{t}_{\delta}))}{2\gamma}+C_d\sqrt{\delta}, \qquad\qquad \forall t>\hat{t}_{\delta} ;
\end{aligned}
\end{equation}

\item \label{p2b}
there exists $\mu > 0$ such that for every $\lambda \in ]0,1[$,
\begin{equation}
\begin{aligned}
\EE\pa{f(X(t))-\min f}&\leq e^{-\mu^2(t-\hat{t}_{\delta})}\EE(f(X(\hat{t}_{\delta}))-\min f)\\
&+e^{-\mu^2 (1-\lambda)(t-\hat{t}_{\delta})}\pa{\frac{LC_{\infty}^2}{2}+C_{**}C_{\infty}\sqrt{\delta}}\\
&+\frac{l_{\delta}(\hat{t}_{\delta}+\lambda(t-\hat{t}_{\delta}))}{\mu^2}+C_f\sqrt{\delta}, \qquad\qquad \forall t>\hat{t}_{\delta}.
\end{aligned}
\end{equation}

Moreover, if \eqref{eq:assumbeta} holds, then
\begin{equation}
\begin{aligned}
\EE\pa{f(X(t))-\min f}&\leq e^{-\mu^2(t-\hat{t}_{\delta})}\EE(f(X(\hat{t}_{\delta}))-\min f) \\
&+e^{-\mu^2 (1-\lambda)(t-\hat{t}_{\delta})}\pa{\frac{LC_{\infty}^2}{2}+C_KC_{\infty}\sqrt{\hat{t}_{\delta}^{\beta-1}}\sqrt{\delta}}\\
&+\frac{k_{\delta}(\hat{t}_{\delta}+\lambda(t-\hat{t}_{\delta}))}{\mu^2}+C_f\sqrt{\delta}, \qquad\qquad \forall t>\hat{t}_{\delta}.
\end{aligned}
\end{equation}
\end{enumerate}

\item \label{pg2} If $p>2$:
\begin{enumerate}[label=(\alph*)]
\item There exists $\gamma>0$ such that
\begin{align}
\EE\pa{\frac{\dist(X(t),\calS)^2}{2}}\leq y_{\delta}^{\star}(t)+C_d\sqrt{\delta}, \qquad\qquad \forall t>\hat{t}_{\delta},
\end{align}
where $y_{\delta}^{\star}$ is the solution of the Cauchy problem 
\begin{enumerate}[label=\rm{(C.\arabic*)}]
\item \label{c1} 
\begin{center}$
\begin{cases}
    y'(t)&=-2^{\frac{p}{2}}\gamma y^{\frac{p}{2}}+h_{\delta}(t),\quad t>\hat{t}_{\delta}\\
    y(\hat{t}_{\delta})&=\EE\pa{\frac{\dist(X(\hat{t}_{\delta},\calS))^2 }{2}\ind_{\Omega_{\delta}}}. \nonumber
\end{cases}$
\end{center}
\end{enumerate}

\item There exists $\mu>0$ such that
\begin{align}
\EE\br{f(X(t))-\min f}\leq w_{\delta}^{\star}(t)+C_f\sqrt{\delta},\qquad\qquad \forall t>\hat{t}_{\delta},
\end{align}
where $w_{\delta}^{\star}$ is the solution of the Cauchy problem
\begin{enumerate}[label=\rm{(C.\arabic*)},start=2]
\item \label{c2}
\begin{center}
$\begin{cases}
    y'(t)&=-\mu^2 y(t)^{2q}+l_{\delta}(t),\quad t>\hat{t}_{\delta}\\
    y(\hat{t}_{\delta})&=\EE([f(X(\hat{t}_{\delta})-\min f]\ind_{\Omega_{\delta}}).
\end{cases}$
\end{center}
\end{enumerate}

\medskip 

Moreover, if \eqref{eq:assumbeta} holds, then 
\begin{align}
\EE\br{f(X(t))-\min f}\leq z_{\delta}^{\star}(t)+C_f\sqrt{\delta}, \qquad\qquad \forall t>\hat{t}_{\delta},
\end{align}
where $z_{\delta}^{\star}$ is the solution of the Cauchy problem
\begin{enumerate}[label=\rm{(C.\arabic*)},start=3]
\item \label{c3}
\begin{center}
$\begin{cases}
y'(t)&=-\mu^2 y(t)^{2q}+k_{\delta}(t),\quad t>\hat{t}_{\delta}\\
y(\hat{t}_{\delta})&=\EE([f(X(\hat{t}_{\delta})-\min f]\ind_{\Omega_{\delta}}).
\end{cases}$
\end{center}
\end{enumerate}
\end{enumerate}
\end{enumerate}
\end{theorem}

\medskip

Before proceeding with the proof, a few remarks are in order.\smallskip
\begin{remark}
The hypothesis that $f$ has a Lipschitz continuous gradient restricts the {\L}ojasiewicz exponent $q$ to be in $[\frac{1}{2},1[$. 
\end{remark}
\begin{remark}
If we have a {\textit{global}} error bound (or {\L}ojasiewicz inequality), then as noted in the discussion of Section~\ref{subsec:locprocess}, one necessarily has $p=2$ (or $q=\frac{1}{2}$). In this case, the statements \ref{p2} of Theorem~\ref{importante3} will hold if we replace $\sigma_{\infty}\in \Lp^2(\R_+)$ by $\sigma_{\infty}$ decreasing and vanishing at infinity, $\delta$ by $0$ and $\hat{t}_{\delta}$ by $0$. Clearly, one recovers \eqref{conv_rate_strongly_convex}.
\end{remark}
\begin{remark}
It is important to highlight the trade-off in the selection of $\delta$. Although $\delta$ can be arbitrarily small, the time from which the inequalities are satisfied, $\hat{t}_{\delta}$, surely increases when $\delta$ approaches $0^+$. Besides, let $q_{\delta,\hat{t}_{\delta}}:\R_+\rightarrow\R$ be a decreasing function. Our convergence rates in Theorem \ref{importante3} are of the form $\EE[m(X(t))]\leq q_{\delta,\hat{t}_{\delta}}(t)+C\sqrt{\delta}, \quad \forall t>t_{\delta}$, where $m(x)=f(x)-\min f$ or $m(x)=\dist(x,\calS)^2/2$. Let $\varepsilon\in]0,2C[$ and  $\delta^{\star}=\frac{\varepsilon^2}{4C^2}$. Then one gets an $\varepsilon$-optimal solution for $t>\max\{q^{\star}(\varepsilon),\hat{t}_{\delta^{\star}}\}$.
\end{remark}
\begin{remark}
Referring again to the discussion of Section~\ref{subsec:locprocess}, we have that there exists $\delta>0$ and $\Omega_{\delta}\in\calF$ with $\Pro(\Omega_{\delta})\geq 1-\delta$ over which we have uniform convergence of the objective. If $\delta$ could be $0$ (a.s. uniform convergence), there would be a $\hat{t}>0$ such that $X(t)\in [f\leq r],\forall t>\hat{t}$  a.s. Thus, the statements in Theorem \ref{importante3} would hold if we replace $\delta$ by $0$ and $\hat{t}_{\delta}$ by $\hat{t}$. The proof is far easier in this case. It is however not easy to ensure the existence of such $\hat{t}$ in general.
\end{remark}
\begin{remark}
In order to find explicit convergence rates in Theorem \ref{importante3} we have to solve or bound the solution of the Cauchy problems \ref{c1}, \ref{c2} and \ref{c3}. We can generalize these problems as follows: Let $a>0,b>1, \hat{t}_{\delta}>0, \delta>0, y_0(\hat{t}_{\delta},\delta)>0$ and $p_{\delta}$ a nonnegative integrable function. Consider 
\begin{enumerate}[label=\rm{(C.0)}]
    \item\label{c0}
    \begin{center}
$ \begin{cases}
    y'(t)&=-ay^b(t)+p_{\delta}(t),\quad t>\hat{t}_{\delta}\\
    y(\hat{t}_{\delta})&=y_0(\hat{t}_{\delta},\delta).
    \end{cases}$
\end{center}
\end{enumerate}

Although one could give an explicit ad-hoc $p_{\delta}$ in order to find a particular solution of \ref{c0}, the dependence of this function on $\hat{t}_{\delta}$ is unavoidable, which is a problem, since $p_{\delta}$ is explicitly related to $\sigma_{\infty}$, and this in turn is the one that defines $\hat{t}_{\delta}$ in the first place.\smallskip

To the best of our knowledge, there is no way to arithmetically solve this non linear ODE, not even a sharp bound of the solution.\smallskip 

Nevertheless, if $y(t)=\calO\pa{(t+1)^{-\frac{1}{b-1}}}$, then $p_{\delta}(t)=\calO\pa{(t+1)^{-\frac{b}{b-1}}}$. Which leads us to make the following conjecture:
\begin{conjecture}\label{conjecture}
\begin{tcolorbox}
\begin{center}

If $p_{\delta}=\calO(\sigma_{\infty}^2)$ and $\sigma_{\infty}^2(t)=\calO\pa{(t+1)^{-\frac{b}{b-1}}}$ (for constants independent of $\delta$ and $\hat{t}_{\delta}$),
then $y(t)=\calO\pa{(t+1)^{-\frac{1}{b-1}}}$.\end{center}
\end{tcolorbox}
\end{conjecture}


\end{remark}

\begin{proof}{Proof of Theorem~\ref{importante3}.}
Given that $\sigma_{\infty}\in\Lp^2(\R_+)$, if it is decreasing, we have immediately that it vanishes at infinity. Let $x^{\star}\in \calS$. Let us recall that by claim \ref{acota} of Theorem~\ref{converge2}, there exists $C^{*}>0$ such that 
\[
\sup_{t\geq 0}\EE\pa{\dist(X(t),\calS)^2} \leq \sup_{t\geq 0}\EE\pa{\norm{X(t)-x^{\star}}^2}\leq C^{*}.
\]
On the other hand, by Theorem~\ref{converge2}\ref{iiconv}, there exists a set $\Omega_{\mathrm{conv}}\in\calF$ such that $\Pro(\Omega_{\mathrm{conv}})=1$ where, for all $\omega\in\Omega_{\mathrm{conv}}$: $\lim_{t\rightarrow\infty} f(X(\omega,t))=\min f$, $t\mapsto f(X(\omega,t))$ is continuous, and $\lim_{t\rightarrow\infty} \dist(X(\omega,t),\calS) = 0$. Then, by Theorem \ref{egorov} for every $\delta>0$ there exists $\Omega_{\delta}\in\calF$ such that $\Omega_{\delta}\subset \Omega_{\mathrm{conv}}$, $\Pro(\Omega_{\delta})>1-\delta$ and $f(X(\cdot,t))$ (resp. $\dist(X(\cdot,t),\calS)$) converges uniformly to $\min f$ (resp. to $0$) on $\Omega_{\delta}$. This means that given $r \geq \min f$, and for every $\delta>0$, there exist $\hat{t}_{\delta}>0$ and $\Omega_{\delta}\in\calF$ with $\Pro(\Omega_{\delta})>1-\delta$ such that $X(\omega,t)\in [f\leq r] \cap \calV_{\calS}$ for all $t\geq \hat{t}_{\delta}$ and $\omega\in\Omega_{\delta}$, where $\calV_{\calS}$ is a neighbourhood of $\calS$. On the other hand, since $f \in \Loj^q(\calS)$, by Proposition~\ref{ebw}, there exists $r>\min f$ and a neighbourhood $\calV_{\calS}$ of $\calS$ such that $f$ verifies the $p$-H\"olderian error bound inequality \eqref{eq:errbnd} on $[\min f < f < r] \cap \calV_{\calS}$. Consequently, for any $\delta > 0$, there exists $t\geq \hat{t}_{\delta}$ large enough such that the $p$-H\"olderian error bound inequality holds at $X(\omega,t)$ for all $t\geq \hat{t}_{\delta}$ and $\omega\in\Omega_{\delta}$.
\medskip
%
%

We are now ready to start. Let $x^{\star}\in \calS$, $\delta>0$, and $t\geq \hat{t}_{\delta}$.

\begin{enumerate}[label=(\roman*)]


\item \label{asin1*} $p=2$: 
\begin{enumerate}[label=(\alph*)]
\item Let $\hat{g}(t)=\hat{\phi}(X(t))=\frac{\dist(X(t),\calS)^2}{2}$, $\hat{G}(t)=\EE(\hat{g}(t)\ind_{\Omega_{\delta}})$, and $\mu>0$ be the coefficient of the error bound inequality. We have 
\[
\nabla\hat{\phi}(X(t))=X(t)-P_\calS(X(t)),
\]
where $P_{\calS}(x)$ is the projection of $x$ on $\calS$, so $\hat{\phi}\in C_1^{1,1}(\R^d)$.  We use Proposition~\ref{itoin} to obtain
\begin{multline}\label{dist*}
\hat{g}(t)-\hat{g}(\hat{t}_{\delta})
\leq - \int_{\hat{t}_{\delta}}^t\dotp{\nabla f(X(s)}{X(s)-P_\calS(X(s))} ds\\
+\int_{\hat{t}_{\delta}}^t \tr[\Sigma(s,X(s))]ds+\int_{\hat{t}_{\delta}}^t\dotp{\sigma^{\top}(s,X(s))(X(s)-P_\calS(X(s)))}{dW(s)}.
\end{multline} 
 
We have that $\tr[\Sigma(s,X(s))]\leq \sigma_{\infty}^2(s)$ and by convexity 
\[
-\dotp{\nabla f(X(s)}{X(s)-P_\calS(X(s))}\leq -\pa{f(X(s))-\min f}.
\]
Therefore,
\begin{multline*}
\hat{g}(t)-\hat{g}(\hat{t}_{\delta})
\leq-\int_{\hat{t}_{\delta}}^t (f(X(s))-\min f) ds\\
+\int_{\hat{t}_{\delta}}^t \sigma_{\infty}^2(s)ds+\int_{\hat{t}_{\delta}}^t\langle \sigma^{\top}(s,X(s))(X(s)-P_\calS(X(s))), dW(s)\rangle.\nonumber
 \end{multline*}
Then, multiplying this inequality by $\ind_{\Omega_{\delta}}$, and taking expectation we obtain  
\begin{multline*}
\hat{G}(t)-\hat{G}(\hat{t}_{\delta})
\leq - \EE\br{\int_{\hat{t}_{\delta}}^t \pa{f(X(s))-\min f} \ind_{\Omega_{\delta}} ds}
+\int_{\hat{t}_{\delta}}^t \sigma_{\infty}^2(s)ds\\
+\EE\br{\ind_{\Omega_{\delta}}\int_{\hat{t}_{\delta}}^t\dotp{\sigma^{\top}(s,X(s))(X(s)-P_\calS(X(s)))}{ dW(s)}}.
\end{multline*}
On the other hand, since $\sigma_{\infty}\in \Lp^2(\R_+)$, we have for all $T>0$
\begin{align*}
\EE\pa{\int_0^T \norm{\sigma^{\top}(s,X(s))(X(s)-P_\calS(X(s)))}^2 ds}&\leq \EE\pa{\int_0^T \sigma_{\infty}^2(s)\norm{X(s)-P_\calS(X(s))}^2 ds}\\
&=\int_0^T \sigma_{\infty}^2(s)\EE(\dist(X(t),\calS)^2)\\
&\leq C^*\int_0^{+\infty} \sigma_{\infty}^2(s)<\infty .
\end{align*}
Letting $Y(s)=\sigma^{\top}(s,X(s))(X(s)-P_\calS(X(s)))$, then 
\[
\EE\br{\int_{\hat{t}_{\delta}}^t\dotp{Y(s)}{dW(s)}}=0.
\] 
This immediately implies
\[
\EE\br{\ind_{\Omega_{\delta}}\int_{\hat{t}_{\delta}}^t\langle Y(s), dW(s)\rangle}=-\EE\br{\ind_{\Omega_{\mathrm{conv}} \setminus \Omega_{\delta}}\int_{\hat{t}_{\delta}}^t\langle Y(s), dW(s)\rangle}.
\]
The right hand side can be bounded using Cauchy-Schwarz inequality as follows
\begin{align*}
&\abs{\EE\br{\ind_{\Omega_{\mathrm{conv}} \setminus \Omega_{\delta}}\int_{\hat{t}_{\delta}}^t\dotp{Y(s)}{dW(s)}}}\\
&=\abs{\EE\br{\ind_{\Omega_{\mathrm{conv}} \setminus \Omega_{\delta}}\int_{\hat{t}_{\delta}}^t\dotp{\sigma^{\top}(s,X(s))(X(s)-P_\calS(X(s)))}{dW(s)}}}\\
&\leq\sqrt{\EE(\ind_{\Omega_{\mathrm{conv}} \setminus \Omega_{\delta}})} \sqrt{\EE\br{\pa{\int_{\hat{t}_{\delta}}^t \dotp{\sigma^{\top}(s,X(s))(X(s)-P_\calS(X(s)))}{dW(s)}}^2}} \\ 
&\leq\sqrt{\delta} \sqrt{\EE\br{\int_{\hat{t}_{\delta}}^t \norm{\sigma^{\top}(s,X(s))(X(s)-P_\calS(X(s)))}^2 ds}}\\
&\leq \sqrt{C^*\delta}\sqrt{\int_{\hat{t}_{\delta}}^t \sigma_{\infty}^2(s) ds}= \sqrt{C^*\delta}\int_{\hat{t}_{\delta}}^t \frac{\sigma_{\infty}^2(s)}{2\sqrt{\int_{\hat{t}_{\delta}}^s \sigma_{\infty}^2(u)du}} ds.  
\end{align*}
Set $C_*=\sqrt{C^*}$, and recall that $C_{\infty}=\sqrt{\int_0^{\infty}\sigma_{\infty}^2(s)ds}$. Thus, for every $t>\hat{t}_{\delta}$
\begin{align}\label{dist2*}
\hat{G}(t)&\leq\hat{G}(\hat{t}_{\delta}) -\int_{\hat{t}_{\delta}}^t\EE\br{(f(X(s))-\min f)\ind_{\Omega_{\delta}}}ds+\int_{\hat{t}_{\delta}}^t\sigma_{\infty}^2(s)ds+C_*\sqrt{\delta} \int_{\hat{t}_{\delta}}^t \frac{\sigma_{\infty}^2(s)}{2\sqrt{\int_{\hat{t}_{\delta}}^s \sigma_{\infty}^2(u)du}} ds.
\end{align}
Recall $h_{\delta}(t)$ from \eqref{eq:hdelta}. Then, we can rewrite \eqref{dist2*} as 
\begin{align}\label{dist2**}
\hat{G}(t)\leq\hat{G}(\hat{t}_{\delta}) -\int_{\hat{t}_{\delta}}^t\EE\br{(f(X(s))-\min f)\ind_{\Omega_{\delta}}}ds+\int_{\hat{t}_{\delta}}^t h_{\delta}(s) ds,\quad \forall t>\hat{t}_{\delta}.
\end{align}
Using that $f\in \EB^2([f\leq r])$, we obtain 
\[
\hat{G}(t)\leq \hat{G}(\hat{t}_{\delta}) -2\gamma\int_{\hat{t}_{\delta}}^t\hat{G}(s)ds+\int_{\hat{t}_{\delta}}^t h_{\delta}(s)ds,\quad \forall t> \hat{t}_{\delta} .
\]
Observe that $h_{\delta}\in \Lp^1([\hat{t}_{\delta},\infty[)$ since
\[
\int_{\hat{t}_{\delta}}^{\infty}h_{\delta}(s)ds\leq C_{\infty}^2+C_*C_{\infty}\sqrt{\delta}.
\] 
The goal now is to apply the comparison lemma to $\hat{G}(t)$ (see Lemma~\ref{comparison}) which necessitates to solve the following ODE 
\begin{equation*}
\begin{cases}
    y'(t)=-2\gamma y(t)+h_{\delta}(t) & t > \hat{t}_{\delta} \\
    y(\hat{t}_{\delta})=\hat{G}(\hat{t}_{\delta}) .
\end{cases}
\end{equation*}
Let $\lambda\in ]0,1[$. Using the integrating factor method, we obtain 
\begin{align*}
y(t)&= e^{-2\gamma(t-\hat{t}_{\delta})}y(\hat{t}_{\delta})+e^{-2\gamma t}\int_{\hat{t}_{\delta}}^{\hat{t}_{\delta}+\lambda(t-\hat{t}_{\delta})}h_{\delta}(s)e^{2\gamma s}ds+e^{-2\gamma t}\int_{\hat{t}_{\delta}+\lambda(t-\hat{t}_{\delta})}^{t} h_{\delta}(s)e^{2\gamma s}ds\\
&\leq e^{-2\gamma(t-\hat{t}_{\delta})}\EE(\hat{g}(\hat{t}_{\delta}))+e^{-2\gamma (1-\lambda)(t-\hat{t}_{\delta})}\int_{\hat{t}_{\delta}}^{\hat{t}_{\delta}+\lambda(t-\hat{t}_{\delta})}h_{\delta}(s) ds\\
&+h_{\delta}(\hat{t}_{\delta}+\lambda(t-\hat{t}_{\delta}))e^{-2\gamma t}\int_{\hat{t}_{\delta}+\lambda(t-\hat{t}_{\delta})}^{t} e^{2\gamma s}ds\\
&\leq e^{-2\gamma(t-\hat{t}_{\delta})}\EE(\hat{g}(\hat{t}_{\delta}))+e^{-2\gamma (1-\lambda)(t-\hat{t}_{\delta})}(C_{\infty}^2+C_*C_{\infty}\sqrt{\delta})+\frac{h_{\delta}(\hat{t}_{\delta}+\lambda(t-\hat{t}_{\delta}))}{2\gamma}.
\end{align*}
where in the first inequality, we used that $\sigma^2$ is decreasing and so is $h_{\delta}$. 
Lemma~\ref{comparison} then gives
\begin{multline*}
\EE\pa{\frac{\dist(X(t),\calS)^2}{2}\ind_{\Omega_{\delta}}}
\leq e^{-2\gamma(t-\hat{t}_{\delta})}\EE\pa{\frac{\dist(X(\hat{t}_{\delta}),\calS)^2}{2}}+e^{-2\gamma (1-\lambda)(t-\hat{t}_{\delta})}(C_{\infty}^2+C_*C_{\infty}\sqrt{\delta})\\
+\frac{h_{\delta}(\hat{t}_{\delta}+\lambda(t-\hat{t}_{\delta}))}{2\gamma}.
\end{multline*}
According to Corollary~\ref{equality} we obtain that for all $t>\hat{t}_{\delta}$
\begin{multline*}
\EE\pa{\frac{\dist(X(t),\calS)^2}{2}}
\leq e^{-2\gamma(t-\hat{t}_{\delta})}\EE\pa{\frac{\dist(X(\hat{t}_{\delta}),\calS)^2}{2}}+e^{-2\gamma (1-\lambda)(t-\hat{t}_{\delta})}(C_{\infty}^2+C_*C_{\infty}\sqrt{\delta})\\
+\frac{h_{\delta}(\hat{t}_{\delta}+\lambda(t-\hat{t}_{\delta}))}{2\gamma}+C_d\sqrt{\delta}.
\end{multline*}

\item \label{asin3*}
Denote $\tilde{g}(t)=\tilde{\phi}(X(t))=f(X(t))-\min f$ and $\tilde{G}(t)=\EE(\ind_{\Omega_{\delta}}\tilde{g}(t))$. By Proposition \ref{itoin}
\begin{multline}
\tilde{g}(t) 
\leq \tilde{g}(\hat{t}_{\delta})-\int_{\hat{t}_{\delta}}^t \dotp{\nabla f(X(s))}{\nabla \tilde{\phi}(X(s))} ds +\frac{L}{2}\int_{\hat{t}_{\delta}}^t \tr[\Sigma(s,X(s))]ds\\
+\ind_{\Omega_{\delta}}\int_{\hat{t}_{\delta}}^t \dotp{\sigma^{\top}(s,X(s))\nabla f(X(s))}{dW(s)} .
\end{multline}
Multiplying both sides by $\ind_{\Omega_{\delta}}$ and taking expectation we obtain
\begin{multline}\label{eq:Gtildebnd}
\tilde{G}(t)-\tilde{G}(\hat{t}_{\delta})
\leq-\EE\br{\int_{\hat{t}_{\delta}}^t \norm{\nabla f(X(s))}^2\ind_{\Omega_{\delta}} ds}
+\frac{L}{2}\EE\br{\int_{\hat{t}_{\delta}}^t \tr[\Sigma(s,X(s))]ds}\\
+\EE\br{\ind_{\Omega_{\delta}}\int_{\hat{t}_{\delta}}^t \dotp{\sigma^{\top}(s,X(s))\nabla f(X(s))}{dW(s)}}.
\end{multline}
On the other hand, we have
\begin{align*}
\EE\pa{\int_0^T \norm{\sigma^{\top}(s,X(s))\nabla f(X(s))}^2 ds}&\leq L^2\EE\pa{\int_0^T \sigma_{\infty}^2(s) \norm{X(s))-x^{\star} }^2 ds}\\
&\leq L^2C^{*}\int_0^{+\infty} \sigma_{\infty}^2(s)<\infty, \quad \forall T>0.
\end{align*}
Since $\EE\br{\int_{\hat{t}_{\delta}}^t\dotp{\sigma^{\top}(s,X(s))\nabla f(X(s))}{dW(s)}}=0$, we have 
\[
\EE\br{\ind_{\Omega_{\delta}}\int_{\hat{t}_{\delta}}^t\dotp{\sigma^{\top}(s,X(s))(\nabla f(X(s)))}{dW(s)}}=-\EE\br{\ind_{\Omega_{\mathrm{conv}} \setminus \Omega_{\delta}}\int_{\hat{t}_{\delta}}^t\dotp{\sigma^{\top}(s,X(s))(\nabla f(X(s))}{dW(s)}}.
\]
The last term can be bounded as
\begin{align*}
&\abs{\EE\br{\ind_{\Omega_{\mathrm{conv}} \setminus \Omega_{\delta}}\int_{\hat{t}_{\delta}}^t\dotp{\sigma^{\top}(s,X(s))(\nabla f(X(s)))}{dW(s)}}}\\
&\leq\sqrt{\EE(\ind_{\Omega_{\mathrm{conv}} \setminus \Omega_{\delta}})} 
\sqrt{\EE\br{\pa{\int_{\hat{t}_{\delta}}^t \dotp{\sigma^{\top}(s,X(s))(\nabla f(X(s)))}{dW(s)}}^2}} \\ 
&\leq L\sqrt{\delta} \sqrt{\EE\br{\int_{\hat{t}_{\delta}}^t \sigma_{\infty}^2(s)\norm{X(s)-x^{\star}}^2 ds }}\\
&\leq L\sqrt{C^{*}}\sqrt{\delta}\sqrt{\int_{\hat{t}_{\delta}}^t \sigma_{\infty}^2(s) ds  }= L\sqrt{C^{*}}\sqrt{\delta}\int_{\hat{t}_{\delta}}^t \frac{\sigma_{\infty}^2(s)}{2\sqrt{\int_{\hat{t}_{\delta}}^s \sigma_{\infty}^2(u)du}} ds.  
\end{align*}
Let us notice that if \eqref{eq:assumbeta} holds, then Proposition~\ref{beta} tells us that $\EE(f(X(t))-\min f)\leq K't^{\beta-1}$ with $\beta\in[0,1[$, and for some $K'>0$. In this case 
\[
\abs{\EE\br{\ind_{\Omega_{\mathrm{conv}} \setminus \Omega_{\delta}}\int_{\hat{t}_{\delta}}^t\dotp{\sigma^{\top}(s,X(s))(\nabla f(X(s)))}{dW(s)}}}
\leq \sqrt{2LK'}\sqrt{\delta}\int_{\hat{t}_{\delta}}^t \frac{\sigma_{\infty}^2(s)s^{\beta-1}}{2\sqrt{\int_{\hat{t}_{\delta}}^s \sigma_{\infty}^2(u)u^{\beta-1}du}} ds.
\]
Injecting this into \eqref{eq:Gtildebnd}, we have for all $ t>\hat{t}_{\delta}$
\begin{multline}\label{dist2**1}
\tilde{G}(t)
\leq\tilde{G}(\hat{t}_{\delta}) -\EE\br{\int_{\hat{t}_{\delta}}^t\norm{\nabla f(X(s))}^2\ind_{\Omega_{\delta}}ds}+\frac{L}{2}\int_{\hat{t}_{\delta}}^t\sigma_{\infty}^2(s)ds \\
+
\begin{cases}
C_K\sqrt{\delta} \int_{\hat{t}_{\delta}}^t \frac{\sigma_{\infty}^2(s)s^{\beta-1}}{2\sqrt{\int_{\hat{t}_{\delta}}^s \sigma_{\infty}^2(u)u^{\beta-1}du}} ds,\quad \forall t>\hat{t}_{\delta} & \text{if \eqref{eq:assumbeta} holds}, \\
C_{**}\sqrt{\delta} \int_{\hat{t}_{\delta}}^t \frac{\sigma_{\infty}^2(s)}{2\sqrt{\int_{\hat{t}_{\delta}}^s \sigma_{\infty}^2(u)du}} ds & \text{otherwise} ,
\end{cases}
\end{multline}
where $C_{**}=L\sqrt{C^{*}}$, $C_K=\sqrt{2LK'}$ and recall that $C_{\infty}=\sqrt{\int_0^{\infty}\sigma_{\infty}^2(s)ds}$. Recalling $l_{\delta}(t)$ and $k_{\delta}(t)$ from \eqref{eq:ldelta}-\eqref{eq:kdelta}, and by Fubini's theorem, \eqref{dist2**1} becomes
\begin{align}\label{dist2**11}
\tilde{G}(t)&\leq\tilde{G}(\hat{t}_{\delta}) -\int_{\hat{t}_{\delta}}^t\EE\br{\norm{\nabla f(X(s))}^2\ind_{\Omega_{\delta}}}ds+
\begin{cases}
\int_{\hat{t}_{\delta}}^t k_{\delta}(s) ds & \text{if \eqref{eq:assumbeta} holds}, \\
\int_{\hat{t}_{\delta}}^t l_{\delta}(s) ds & \text{otherwise} .
\end{cases}
\end{align}
Since $f \in \Loj^{1/2}(\calS)$, there exists $\mu>0$ such that 
\begin{align}\label{dist2f}
\tilde{G}(t)&\leq\tilde{G}(\hat{t}_{\delta}) -\mu^2\int_{\hat{t}_{\delta}}^t \tilde{G}(s) ds+
\begin{cases}
\int_{\hat{t}_{\delta}}^t k_{\delta}(s) ds & \text{if \eqref{eq:assumbeta} holds}, \\
\int_{\hat{t}_{\delta}}^t l_{\delta}(s) ds & \text{otherwise} .
\end{cases}
\end{align}
To get an explicit bound in \eqref{dist2f}, we use Lemma~\ref{comparison}, which involves solving
\begin{enumerate}[label=(E.\arabic*),start=2]
\item \label{e2}\begin{center}
$\begin{cases}
    y'(t)&=-\mu^2 y(t)+l_{\delta}(t),\quad t>\hat{t}_{\delta}\\
    y(\hat{t}_{\delta})&=\tilde{G}(\hat{t}_{\delta})
\end{cases}$
\end{center}

\item \label{e3}\begin{center}
$\begin{cases}
    y'(t)&=-\mu^2 y(t)+k_{\delta}(t),\quad t>\hat{t}_{\delta}\\
    y(\hat{t}_{\delta})&=\tilde{G}(\hat{t}_{\delta})
\end{cases}$
\end{center}
\end{enumerate}

Let $\lambda\in ]0,1[$. Using the integrating factor method as in \ref{asin1*}, we get for \ref{e2}
\begin{align*}
y(t)
\leq 
e^{-\mu^2(t-\hat{t}_{\delta})}\EE(\tilde{g}(\hat{t}_{\delta})) +
\begin{cases}
e^{-\mu^2 (1-\lambda)(t-\hat{t}_{\delta})}\pa{\frac{LC_{\infty}^2}{2}+C_{**}C_{\infty}\sqrt{\delta}}+\frac{l_{\delta}(\hat{t}_{\delta}+\lambda(t-\hat{t}_{\delta}))}{\mu^2} & \text{for \ref{e2}}\\
e^{-\mu^2 (1-\lambda)(t-\hat{t}_{\delta})}\pa{\frac{LC_{\infty}^2}{2}+C_KC_{\infty}\sqrt{\hat{t}_{\delta}^{\beta-1}}\sqrt{\delta}} + \frac{k_{\delta}(\hat{t}_{\delta}+\lambda(t-\hat{t}_{\delta}))}{\mu^2} & \text{for \ref{e3}} .
\end{cases}
\end{align*}
Using Lemma~\ref{comparison} and Corollary~\ref{equality}
\begin{align*}
\EE\br{f(X(t))-\min f}
&\leq y(t) + C_f\sqrt{\delta} \\
&\leq e^{-\mu^2(t-\hat{t}_{\delta})}\EE\br{f(X(\hat{t}_{\delta}))-\min f}+C_f\sqrt{\delta} \\
&+
\begin{cases}
e^{-\mu^2 (1-\lambda)(t-\hat{t}_{\delta})}\pa{\frac{LC_{\infty}^2}{2}+C_{**}C_{\infty}\sqrt{\delta}}+\frac{l_{\delta}(\hat{t}_{\delta}+\lambda(t-\hat{t}_{\delta}))}{\mu^2} & \text{for \ref{e2}}\\
e^{-\mu^2 (1-\lambda)(t-\hat{t}_{\delta})}\pa{\frac{LC_{\infty}^2}{2}+C_KC_{\infty}\sqrt{\hat{t}_{\delta}^{\beta-1}}\sqrt{\delta}} + \frac{k_{\delta}(\hat{t}_{\delta}+\lambda(t-\hat{t}_{\delta}))}{\mu^2} & \text{for \ref{e3}} .
\end{cases}
\end{align*}
\end{enumerate}

\item $p > 2$:
\begin{enumerate}[label=(\alph*)]
\item We embark from inequality \eqref{dist2**} and we now use that $f\in \EB^p([f\leq r])$ with $p>2$, to get
\begin{align}
\hat{G}(t)&\leq\hat{G}(\hat{t}_{\delta}) -\int_{\hat{t}_{\delta}}^t\EE\br{(f(X(s))-\min f)\ind_{\Omega_{\delta}}}ds+\int_{\hat{t}_{\delta}}^t h_{\delta}(s) ds\\
&\leq \hat{G}(\hat{t}_{\delta}) -2^{p/2}\gamma\int_{\hat{t}_{\delta}}^t\hat{G}(s)^{p/2}+\int_{\hat{t}_{\delta}}^t h_{\delta}(s) ds.\nonumber
\end{align}
In the last inequality, we used that $p>2$ and Jensen's inequality.

\smallskip

The idea is again to use the comparison lemma (Lemma~\ref{comparison}), which will now involve solving the Cauchy problem \ref{c1}, and finally invoke Corollary~\ref{equality}. 
%
%
\item The reasoning is similar to the previous point using now that $f \in \Loj^{q}(\calS)$ and the computations of \ref{asin1*}\ref{asin3*}. We omit the details for the sake of brevity.
\end{enumerate}
\end{enumerate}
\end{proof}

%% file: tex/Nonsmooth_structured.tex
\section{SDE for nonsmooth structured convex optimization}\label{nonsmooth_structured}

In this section, we turn to the composite convex minimization problem with additive structure
\begin{equation}\label{P01}
    \min_{x\in\R^d} f(x)+g(x),
\end{equation}
where 
\begin{align}\label{H0'}
\begin{cases}
\text{$f \in C^{1,1}_L(\R^d) \cap \Gamma_0(\R^d)$ and $g \in \Gamma_0(\R^d)$}; \\
\calS = \argmin(f+g) \neq \emptyset. \tag{$\mathrm{H}_0'$} 
\end{cases}
\end{align} 
The importance of this class of problems comes from its wide spectrum of applications ranging from data processing, to machine learning and statistics to name a few. 
%
\smallskip

\noindent We consider two different approaches leading to different SDE's. The first is based on a fixed point argument and the use of the notion of cocoercive monotone operator. The second approach is based on a regularization/smoothing argument, for instance the Moreau envelope.

\subsection{Fixed point approach via cocoercive monotone operators}

Let us start with some classical definitions concerning monotone operators.

\begin{definition}
An operator $A:\R^d \to \calP(\R^d)$ is monotone if
\[
\dotp{u-v}{x-y} \geq 0, \quad \quad \forall (x,u) \in \gra(A), (y,v) \in \gra(A).
\]
It is maximally monotone if there exists no monotone operator whose graph properly contains $\gra(A)$.
Moreover, $A$ is $\gamma-$strongly monotone with modulus $\gamma > 0$ if
\[
\dotp{u-v}{x-y}  \geq \gamma\norm{x-y}^2, \quad \forall (x,u) \in \gra(A), (y,v) \in \gra(A).
\]
\end{definition}
\begin{remark}
If $A$ is maximally monotone and strongly monotone, then $A^{-1}(0)\eqdef  \{x\in\R^d: A(x)=0\}$ is non-empty and reduced to a singleton.
\end{remark}

\begin{remark}
The subdifferential operator $\partial g$ of $g \in \Gamma_0(\R^d)$ is maximally monotone.
\end{remark}

\begin{definition}
A single-valued operator $M:\R^d\rightarrow\R^d$ is cocoercive with constant $\rho>0$ if 
\[
\langle M(x)-M(y),x-y\rangle \geq \rho\norm{M(x)-M(y)}^2,\quad \forall x,y\in \R^d.
\]
\end{definition}
\begin{remark}
It is clear that a cocoercive operator is $\rho^{-1}-$Lipschitz continuous. In turn, a cocoercive operator is maximally monotone. 
\end{remark}

\begin{remark}
If $f \in C_L^{1,1}(\R^d)\cap \Gamma_0(\R^d)$, then the operator $\nabla f$ is $L^{-1}-$cocoercive.
\end{remark}

\medskip

Our interest now is to solve the structured monotone inclusion problem 
\[
0 \in A(x)+B(x), 
\]
where $A$ is maximally monotone, and $B$ is cocoercive with $(A+B)^{-1}(0) \neq \emptyset$. This is of course a generalization of \eqref{P01} by taking $A=\partial g$ and $B=\nabla f$.

A favorable situation occurs when one can compute the resolvent operator of $A$
\[
J_{\mu A} = (I + \mu A )^{-1},  \quad \mu > 0 .
\]
In this case, we can develop a strategy parallel to the one which consists in replacing a maximally monotone operator by its Yosida approximation. 
Indeed, given $\mu>0$, we have
\begin{equation}\label{eq.zero_operateur}
(A+B) (x) \ni 0 \iff x- J_{\mu A}( x-\mu B(x)) =0\iff M_{A,B,\mu}(x) =0,
\end{equation}
where $M_{A,B,\mu}: \R^d \to \R^d$ is the single-valued operator defined by 
\begin{equation}\label{Max_Mon_structured}
M_{A,B,\mu}(x) = \frac{1}{\mu} \pa{x- J_{\mu A}( x-\mu B(x))}.
\end{equation}

$M_{A,B,\mu}$ is closely tied to the well-known forward-backward fixed point operator. Moreover, when $B=0$, $M_{A,B,\mu}=\frac{1}{\mu} \pa{I - J_{\mu A}}$ which is nothing but the Yosida regularization of $A$ with index $\mu$. As a remarkable property, for the $\mu$ parameter properly set, the operator $M_{A,B,\mu}$ is cocoercive. This is made precise in the following result. 


\begin{proposition}\label{cocoer} \cite[Lemma~B.1]{cocoercive} Let
$A:\R^d\to \calP(\R^d)$ be a general maximally monotone operator, and let  $B:\R^d\to \R^d$ be a monotone operator which is $\lambda$-cocoercive. Assume that $\mu  \in ]0, 2\lambda[ $. Then,
$M_{A,B,\mu}$ is $\rho$-cocoercive with 
\[
\rho =\mu\pa{1-\frac{\mu}{4\lambda}}.
\]
\end{proposition}

\medskip

We first focus on finding the zeros of $M$, where
\begin{align}
M:\R^d\rightarrow\R^d \text{ is cocoercive and } M^{-1}(0) \neq \emptyset. \tag{$\mathrm{H}_{0}^M$}\label{H0M}
\end{align} 
We will then specialize our results to the case of a structured operator of the form $M_{A,B,\mu}$.

\medskip

Our goal is to handle the situation where $M$ can be evaluated up to a stochastic error. We therefore consider the following SDE, defined for (deterministic) initial data $X_0\in\R^d$,
\begin{equation} \label{CSGD1}\tag{$\mathrm{SDE}^M$}
\begin{cases}
\begin{aligned}
dX(t)&=-M(X(t))dt+\sigma(t,X(t))dW(t), \quad t\geq 0\\
X(0)&=X_0 .
\end{aligned}
\end{cases}
\end{equation}
As in Section~\ref{problem_statement}, we will assume that $W$ is a $\calF_t$-adapted $m-$dimensional Brownian motion, and the volatility matrix $\sigma:\R_+\times\R^d\rightarrow\R^{d\times m}$ satisfies \eqref{H}.

\smallskip

Let us now state the natural extensions of our main results to this situation.
\begin{theorem}\label{converge21}
Let $M:\R^d\rightarrow\R^d$ be a cocoercive operator. Consider the stochastic differential equation \eqref{CSGD1} under the hypotheses \eqref{H0M} and \eqref{H}. Then, there exists a unique solution $X\in S_d^{\nu}$, for every $\nu\geq 2$.  Moreover, if $\sigma_{\infty}\in \Lp^2(\R_+)$, then: 
 \begin{enumerate}[label=(\roman*)]
     \item $\sup_{t\geq 0}\EE[\norm{X(t)}^2]<\infty$.
     
     \vspace{1mm}
    
     \item  $\forall x^{\star}\in M^{-1}(0)$, $\lim_{t\rightarrow\infty} \norm{X(t)-x^{\star}}$ exists a.s. and $\sup_{t\geq 0}\norm{X(t)}<\infty$ a.s.
     
       \vspace{1mm}
       
     \item  \label{iiconv1}
     $\lim_{t\rightarrow \infty}\norm{M(X(t))}=0$ a.s.
     
       \vspace{1mm}
     \item  There exists an $M^{-1}(0)-$valued random variable $x^{\star}$ such that $\lim_{t\rightarrow\infty} X(t) = x^{\star}$ a.s.
 \end{enumerate}
\end{theorem}

\begin{proof}
Existence and uniqueness follow from Theorem~\ref{teoexistencia} since $M$ is Lipschitz continuous and $\sigma$ verifies \eqref{H}. The proof of the first two items remains the same as for Theorem \ref{converge2}. For the third item, we use the cocoercivity of $M$ instead of the convexity of $f$ and Corollary \ref{2L} to prove that $\lim_{t\rightarrow \infty}\norm{M(X(t))}=0$ a.s.  For the last item,  it suffices to use that the operator $M$ is continuous (since it is Lipschitz continuous) to conclude with Opial's Lemma.
\end{proof}

\begin{theorem}\label{importante01}
Let $M:\R^d\rightarrow\R^d$ be a $\rho-$cocoercive operator. Let us make the assumptions \eqref{H0M} and \eqref{H}. Let  $X\in S_d^2$ be the solution of \eqref{CSGD1} with initial condition $X_0$. Then the following properties are satisfied:
\begin{enumerate}[label=(\roman*)]
\item \label{0i1} Let $\overline{M\circ X}(t)\eqdef t^{-1}\int_0^t M(X(s))ds$ and $\overline{\norm{M(X(t))}^2}\eqdef t^{-1}\int_0^t \norm{M(X(s))}^2ds$. We have
\begin{equation}\label{eq:rateMbar}
\EE\br{\norm{\overline{M\circ X}(t)}^2}\leq 
\EE\br{\overline{\norm{M(X(t))}^2}}\leq  \frac{\dist(X_0,M^{-1}(0))^2}{2\rho t}+\frac{\sigma_*^2}{2\rho}
,\quad \forall t> 0.
\end{equation}
Besides, if $\sigma_{\infty}$ is $\Lp^2(\R_+)$, then 
\begin{equation}\label{eq:rateMbarL2}
\EE\br{\norm{\overline{M\circ X}(t)}^2}\leq 
\EE\br{\overline{\norm{M(X(t))}^2}}=\calO\pa{\frac{1}{t}}, \quad \forall t> 0.
\end{equation}

\item \label{0ii1} If $M$ is $\gamma-$strongly monotone, then $M^{-1}(0)=\{x^{\star}\}$ and 
\begin{equation}
\EE\pa{\frac{\norm{X(t)-x^{\star}}^2}{2}}\leq \frac{\norm{X_0-x^{\star}}^2}{2}e^{-2\gamma t}+\frac{\sigma_*^2}{4\gamma}
,\quad \forall t\geq 0.\end{equation}
If, moreover, $\sigma_{\infty}$ is decreasing and vanishes at infinity, then for every $ \lambda\in ]0,1[$
\begin{equation}
\EE\pa{\frac{\norm{X(t)-x^{\star}}^2}{2}}\leq \frac{\norm{X_0-x^{\star}}^2}{2}e^{-2\gamma t}+\frac{\sigma_*^2}{4}e^{-2\gamma t(1-\lambda)}+\frac{\sigma_{\infty}^2(\lambda t)}{2}, \quad \forall t > 0.
\end{equation}
\end{enumerate}
\end{theorem}

\begin{proof}
Analogous to Theorem~\ref{importante0}.
\end{proof}

\medskip

We now turn to the local convergence properties. To this end, we need an extension of the H\"olderian error bound inequality (or {\L}ojasiewicz inequality) to the operator setting. For convex functions, it is known that error bound inequalities are closely related to metric subregularity of the subdifferential \cite{Geoffroy08,Kruger15a,Kruger15b}. This leads to the following definition.
\begin{definition} 
Let $M:\R^d\rightarrow\R^d$ be a single-valued operator. We say that $M$ satisfies the H\"older metric subregularity property with exponent $p\geq 2$ at $x^{\star}\in M^{-1}(0)$ if there exists $\gamma > 0$ and a neighbourhood $\calV_{x^{\star}}$ such that
\begin{equation}
\norm{M(x)}^2\geq \gamma\dist(x,M^{-1}(0))^p, \quad \forall x\in \calV_{x^{\star}}.
\end{equation}
If this inequality holds for any $x^{\star}\in M^{-1}(0)$ with the same $\gamma$, we will write $M\in \mathrm{HMS}^p(\R^d)$.
\end{definition}

\begin{theorem}\label{importante31}
Let $M$ be a $\rho-$cocoercive operator such that $M\in \mathrm{HMS}^2(\R^d)$. Let $X\in S_d^2$ be the solution of  \eqref{CSGD1} under the hypotheses \eqref{H0M}, \eqref{H}. Suppose  that $\sigma_{\infty}\in \Lp^2(\R_+)$ ($C_{\infty}\eqdef \norm{\sigma_{\infty}}_{\Lp^2(\R_+)}$) and $\sigma_{\infty}$ is decreasing. 
Consider also the positive constants $C,C_d,\gamma$. Then, for all $\delta>0$, there exists $\hat{t}_{\delta}>0$ such that for every $\lambda \in (0,1)$: 
\begin{align}
\EE\pa{\frac{\dist(X(t),M^{-1}(0))^2}{2}}&\leq e^{-2\gamma\rho(t-\hat{t}_{\delta})}\EE\pa{\frac{\dist(X(\hat{t}_{\delta}),M^{-1}(0))^2}{2}}\nonumber\\
&+e^{-2\gamma\rho (1-\lambda)(t-\hat{t}_{\delta})}(C_{\infty}^2+C_{\infty}C\sqrt{\delta})\\
&+\frac{h_{\delta}(\hat{t}_{\delta}+\lambda(t-\hat{t}_{\delta}))}{2\gamma\rho}+C_d\sqrt{\delta}, \quad \forall t>\hat{t}_{\delta},\nonumber
\end{align}
where $h_{\delta}(t)=\sigma_{\infty}^2(t)+C\sqrt{\delta}  \frac{\sigma_{\infty}^2(t)}{2\sqrt{\int_{\hat{t}_{\delta}}^t \sigma_{\infty}^2(u)du}}$.
\end{theorem}
\begin{proof}
The proof is essentially the same as that of Theorem~\ref{importante3}\ref{p2}\ref{p2a}, where instead of convexity in \eqref{dist*}, we use cocoercivity of $M$, and in \eqref{dist2**} we invoke Theorem~\ref{converge21} and H\"older metric subregularity.
\end{proof}

\begin{remark}
We can naturally extend the previous result for $p>2$ as in Theorem \ref{importante3}\ref{pg2}. Nevertheless, since that bound is not explicit, we will skip this extension. 
\end{remark}

\medskip

As an immediate consequence of the above result, by considering the cocoercive operator 
$M_{A,B,\mu}$ defined in \eqref{Max_Mon_structured}, we obtain the following result.

\begin{corollary}
Let $A:\R^d\to \calP(\R^d)$ be a maximally monotone operator and $B:\R^d\to \R^d$ be a $\lambda$-cocoercive operator, $\lambda > 0$. Let $M_{A,B,\mu}$ be the operator defined in \eqref{Max_Mon_structured}. 
Assume that $\mu  \in ]0, 2\lambda[ $ and $(A+B)^{-1}(0) \neq \emptyset$. Then, 
the operator $M_{A,B,\mu}$ is $\rho$-cocoercive with $\rho =\mu\pa{1-\frac{\mu}{4\lambda}}$,
and the SDE:
\begin{equation}
\begin{cases} 
\begin{aligned} 
dX(t)&=-M_{A,B,\mu}(X(t))dt+\sigma(t,X(t))dW(t), \quad t\geq 0\\
X(0)&=X_0,\nonumber
\end{aligned}
\end{cases}
\end{equation}
has a unique solution $X\in S_d^{\nu}$, for every $\nu\geq 2$, that verifies the conclusions of Theorem~\ref{converge21} and Theorem~\ref{importante01}. In particular, if $\sigma_{\infty}\in\Lp^2(\R_+)$, there exists an $(A+B)^{-1}(0)-$valued random variable $x^{\star}$ such that $\lim_{t\rightarrow\infty} X(t) = x^{\star}$ a.s.
\end{corollary}

This result naturally applies to problem \eqref{P01} when $\calS = \argmin(f+g) \neq \emptyset$ by taking $A=\partial g$ and $B=\nabla f$. In this case, one has that $X(t)$ converges a.s. to an $\calS$-valued random variable. Moreover, using standard inequalities, see \eg \cite{fista2009}, one can show that
\[
\EE\br{(f+g)\pa{t^{-1}\int_{0}^t \pa{\prox_{\mu g}(x-\mu \nabla f (x))}ds}-\min(f+g)} = \calO\pa{\sqrt{\EE\br{\overline{\norm{M(X(t))}^2}}}},
\]
where $\prox_{\mu g}=(I+\mu \nabla g)^{-1}$ is the proximal mapping of $g$. From this, one can deduce an $\calO(t^{-1/2})$ rate thanks to \eqref{eq:rateMbar} and \eqref{eq:rateMbarL2}.


\subsection{Approach via Moreau-Yosida regularization}\label{sec:nonsmooth}

The previous approach, though it is able to deal with more general setting (that of monotone inclusions), took us out of the framework of convex optimization by considering instead a dynamic governed by a cocoercive operator. In particular, the perturbation/noise is considered on the whole operator evaluation and not on a part of it (\ie $B$) as it is standard in many applications. Moreover this approach led to a pessimistic convergence rate estimate when specialized to convex function minimization. By contrast, the following approach will operate directly on problem \eqref{P01} and is based on a standard smoothing approach, replacing the non-smooth part $g$ by its Moreau envelope \cite{smoothing}.

\subsubsection{Moreau envelope}
Let us start by recalling some basic facts concerning the Moreau envelope.
\begin{definition}
Let $g \in \Gamma_0(\R^d)$. Given $\theta>0$, the Moreau envelope of $g$ of parameter $\theta$ is the function
\[
g_{\theta}(x)\eqdef \inf_{y\in\R^d} \pa{g(y)+\frac{1}{2\theta}\norm{x-y}^2}=\pa{g ~\square~ \frac{1}{\theta}q}(x)
\]
where $\square$ is the infimal convolution operator and $q(x)=\frac{1}{2}\norm{x}^2$.
\end{definition}

The Moreau envelope has remarkable approximation and regularization properties, as summarized in the following statement.
\begin{proposition}\label{des}
Let $g \in \Gamma_0(\R^d)$.
\begin{enumerate}[label=(\roman*)]
     \item \label{i} $g_{\theta}(x) \downarrow \inf g(\R^d)$ as $\theta\uparrow +\infty$.
     \smallskip
     \item \label{ii} $g_{\theta}(x) \uparrow  g(x)$ as $\theta\downarrow 0$.
     \smallskip
     \item \label{iii} $g_{\theta}(x)\leq g(x)$ for any $\theta>0$ and $x\in \R^d$, 
     \smallskip
     \item \label{iv} $\argmin(g_{\theta})=\argmin(g)$ for any $\theta>0$, 
     \smallskip
     \item \label{v} $g(x) -g_{\theta}(x)\leq \frac{\theta}{2} \norm{\partial^0 g(x)}^2$ 
     for any $\theta>0$ and $x \in \dom (\partial g)$,
     \item \label{prim} $g_{\theta}\in C_{\frac{1}{\theta}}^{1,1}(\R^d)\cap\Gamma_0(\R^d)$ for any $\theta>0$. 
 \end{enumerate}
\end{proposition}

We use the following notation in the rest of the section: $F\eqdef f+g, \calS \eqdef \argmin F$, $F_{\theta} \eqdef f+g_{\theta}$ and $\calS_{\theta}\eqdef \argmin F_{\theta}$.

\smallskip

Note that  $F_{\theta} \in C_{L+\frac{1}{\theta}}^{1,1}(\R^d)\cap\Gamma_0(\R^d)$. Thus we will use $F_{\theta}$ as the potential driving \eqref{CSGD}, that is 
\begin{equation} \label{CSGD_theta}\tag{$\mathrm{SDE}_{\theta}$}
\begin{cases}
\begin{aligned}
dX(t)&=-\nabla F_{\theta}(X(t))dt+\sigma(t,X(t))dW(t), \;t\geq 0\\
X(0)&=X_0 .
\end{aligned}
\end{cases}
\end{equation}
Under \eqref{H0'} and \eqref{H}, we will show almost sure convergence of the trajectory and corresponding convergence rates.

\begin{remark}
Though we focus here on the Moreau envelope, our convergence results, in particular, Proposition~\ref{nuevo1}, still hold with infimal-convolution based smoothing using more general smooth kernels beyond the norm squared; see \cite[Section~4.4]{smoothing}.
\end{remark}

\subsubsection{Convergence of the trajectory}

Applying Theorem~\ref{converge2} to $F_{\theta}$, we have the following result.

\begin{proposition}\label{Moreau_conv_1}
For any $\theta > 0$, let $X_{\theta}\in S_d^2$ be the solution of the dynamic \eqref{CSGD_theta} governed by the potential $F_{\theta}$, and make assumptions \eqref{H0'}, $\calS_{\theta}\neq\emptyset$, \eqref{H} and $\sigma_{\infty}\in\Lp^2(\R_+)$. Then there exists an $\calS_{\theta}$-valued random variable $x_{\theta}^{\star}$  such that 
\[
\lim_{t\rightarrow\infty} X_{\theta}(t)=x_{\theta}^{\star}, \quad a.s.
\]
\end{proposition}
\if
{
\begin{remark}
If $f=0$, then $\calS_{\theta}=\calS$. Else, by \cite[Theorem 7.33]{vari}, we have that
\[
\limsup_{\theta\rightarrow 0^+} \calS_{\theta}\subset \calS.
\]
 Let $\Omega_{\theta}\in\calF$ such that $\Pro(\Omega_{\theta})=1$ and $\lim_{t\rightarrow\infty} X_{\theta}(\omega,t)=x_{\theta}^{\star}(\omega), \forall \omega\in\Omega_{\theta}$.
Although one could think that $\limsup_{\theta\rightarrow 0^+} \lim_{t\rightarrow\infty} X_{\theta}(t)$ is an $\calS-$random variable, it is not clear how to ensure that we can take the $\limsup$ over $\theta$ going to $0^+$, since we cannot prove that $\Pro\pa{\bigcap_{\theta\in (0,1)}\Omega_{\theta}}=1$ because the usual covering arguments fails due to the fact that $(0,1)$ is an open set.
\end{remark}
}
\fi


If $f=0$, then $\calS_{\theta}=\calS$ (see Proposition~\ref{des}\ref{iv}), and Proposition \ref{Moreau_conv_1} provides almost sure convergence to a solution of \eqref{P01}. On the other hand for $f\neq 0$, $\calS \neq \calS_\theta$ in general and we only obtain an "approximate" solution of \eqref{P01}; see Proposition~\ref{pp}\ref{pp:claim2} for a quantitative estimate of this approximation when $f$ is strongly convex. To obtain a true solution of the initial problem, a common device consists in using a diagonalization process which combines the dynamic with the approximation. Specifically, one considers
\begin{equation} \label{CSGD_theta_t}\tag{$\mathrm{SDE}_{\theta(t)}$}
\begin{cases}
\begin{aligned}
dX(t)&=-\nabla F_{\theta (t)}(X(t))dt+\sigma(t,X(t))dW(t), \;t\geq 0\\
X(0)&=X_0 ,
\end{aligned}
\end{cases}
\end{equation}
where $\theta(t) \downarrow 0$ as $t\to +\infty$. 
In the deterministic case, an abundant literature has been devoted to the convergence of this type of systems. Note that unlike the cocoercive approach, we are now faced with a non-autonomous stochastic differential equation, making this a difficult problem, a subject for further research.

\subsubsection{Convergence rates}
We start with the following uniform bound on $\calS_\theta$ which holds under slightly reinforced, but reasonable assumptions on $f$ and $g$.
\begin{proposition}\label{pp}
Consider $f,g$ where $f$ and $g$ and are proper lsc and convex, and $g$ is also $L_0$-Lipschitz continuous. 
\begin{enumerate}[label=(\roman*)]
\item \label{pp:claim1}
Assume that $F=f+g$ is coercive. Then for any $\theta \geq 0$ there exists $C>0$ (independent of $\theta$) such that 
\begin{equation}\label{eq:unifbnd}
\sup_{z\in \calS_{\theta}}\norm{z}\leq C.
\end{equation}
\item \label{pp:claim2}
Assume that $f\in \Gamma_{\mu}(\R^d)$ for $\mu > 0$, then \eqref{eq:unifbnd} holds, $\calS=\{x^{\star}\}$, $\calS_{\theta}=\{x_{\theta}^{\star}\}$ and
\begin{equation}\label{eq:errminmizers}
\norm{x_{\theta}^{\star}-x^{\star}}^2\leq \frac{L_0}{\mu}\theta .
\end{equation}
\end{enumerate}
\end{proposition}

\begin{proof}
\begin{enumerate}[label=(\roman*)]
\item Since $F$ is coercive, so is $F_\theta$. Thus both $\calS$ and $\calS_\theta$ are non-empty compact sets. Let $x_{\theta}^{\star}\in \calS_{\theta}$ and $x^{\star}\in \calS$. By Proposition~\ref{des}\ref{v} and Lipschitz continuity of $g$, we obtain 
\[
F(x_{\theta}^{\star})\leq F_{\theta}(x_{\theta}^{\star})+\frac{L_0^2}{2}\theta.
\]
Moreover, 
\[
F_{\theta}(x_{\theta}^{\star})+\frac{L_0^2}{2}\theta\leq F_{\theta}(x^{\star})+\frac{L_0^2}{2}\theta\leq F(x^{\star})+\frac{L_0^2}{2}\theta\leq \min(F)+\frac{L_0^2}{2}
\eqdef \tilde{C},
\]
where the second inequality is given by Proposition~\ref{des}\ref{iv}. On the other hand, the coercivity of $F$ implies that there exists $a>0,b\in\R$ such that for any $x \in \R^d$ 
\[
a\norm{x}+b\leq F(x).
\]
Therefore, collecting the above inequalities yields
\[
a\norm{x_{\theta}^{\star}}+b\leq F(x_{\theta}^{\star})\leq \tilde{C}.
\]
Taking the supremum over $x_{\theta}^{\star}$ and defining $C\eqdef \frac{\tilde{C}-b}{a}\geq 0$, we obtain \eqref{eq:unifbnd}, or equivalently that the set of approximate minimizers is bounded independently of $\theta$. 

\item Since $f$ is $\mu$-strongly convex, so are $F$ and $F_{\theta}$. In turn, $F$ is coercive and thus \eqref{eq:unifbnd} holds by claim \ref{pp:claim1}. Strong convexity implies uniqueness of minimizers of $F$ and $F_\theta$. Moreover, 
\begin{equation}\label{eq:strconv1}
\frac{\mu}{2}\norm{x_{\theta}^{\star}-x^{\star}}^2\leq F_{\theta}(x^{\star})-F_{\theta}(x_{\theta}^{\star}).
\end{equation}
From Proposition~\ref{des}\ref{iii}-\ref{v} and  and Lipschitz continuity of $g$, we infer that
\begin{equation}\label{eq:strconv2}
F_{\theta}(x^{\star})-F_{\theta}(x_{\theta}^{\star})\leq F(x^{\star})-F_{\theta}(x_{\theta}^{\star})\leq F(x_{\theta}^{\star})-F_{\theta}(x_{\theta}^{\star})=g(x_{\theta}^{\star})-g_{\theta}(x_{\theta}^{\star})\leq \frac{L_0}{2}\theta.
\end{equation}
Combining \eqref{eq:strconv1} and \eqref{eq:strconv1}, we get the claimed bound.
\end{enumerate}
\end{proof}


We are now ready to establish complexity results.

\begin{proposition}\label{nuevo1}
Suppose that in addition to \eqref{H0'} and \eqref{H}, $F=f+g$ is coercive and $g$ is $L_0$-Lipschitz continuous. Let $X_{\theta}$ be the solution of \eqref{CSGD_theta} governed by $F_{\theta}$ with $\theta>0$. Let $C_0=\norm{X_0}+C$, where $C$ is the constant (independent of $\theta$), defined in \eqref{eq:unifbnd}. Then the following statements hold for any $t > 0$.
\begin{enumerate}[label=(\roman*)]
\item \label{nuevo1:claim1}
Let $\displaystyle\overline{X_\theta}(t)=t^{-1}\int_0^t X_\theta(s)ds$, then 
\[
\EE\pa{F\pa{\overline{X_{\theta}}(t)} - \min F} \leq \frac{C_0^2}{2t}+\frac{\sigma_*^2}{2}+\theta \frac{L_0^2}{2} .
\]

Besides, if $\sigma_{\infty}\in \Lp^2(\R_+)$, then 
\[
\EE\pa{F\pa{\overline{X_{\theta}}(t)} - \min F} = \frac{C_0^2+\int_0^{+\infty}\sigma_{\infty}^2(s) ds}{2t}+\theta \frac{L_0^2}{2} .
\]

\item If $\sigma_{\infty}$ verifies \eqref{eq:assumbeta} and $\theta \in ]0,1]$, then 
\[
\EE\pa{F(X(t)) -\min F} = \frac{C_0^2}{2t}+\frac{K(1+L)}{2\theta}t^{\beta-1} + \theta \frac{L_0^2}{2} .
\]
   
\item \label{nuevo1:claim2}
If, in addition, $f\in \Gamma_{\mu}(\R^d)$ for some $\mu > 0$ then $\calS=\{x^{\star}\}$, $\calS_{\theta}=\{x_{\theta}^{\star}\}$, and
    \[
    \EE\pa{\norm{X_{\theta}(t)-x^{\star}}^2}\leq 2C_0^2e^{-2\mu t}+\frac{\sigma_*^2}{\mu}+2\frac{L_0}{\mu}\theta .
    \]
    
%
                  
Besides, if $\sigma_{\infty}$ is decreasing and vanishes at infinity, then $\forall \lambda\in ]0,1[$:
\[
\EE\pa{\norm{X_{\theta}(t)-x^{\star}}^2}\leq 2C_0^2e^{-2\mu t}+\frac{\sigma_*^2}{\mu}e^{-2\mu) (1-\lambda)t}+2\sigma_{\infty}^2(\lambda t)+2\frac{L_0}{\mu}\theta .
\]
\end{enumerate}
\end{proposition}

\begin{remark}
Observe that when $f=0$, then $\calS_{\theta}=\calS$. Therefore in Proposition~\ref{nuevo1} we have $x_{\theta}^{\star}=x^{\star}$ and the last term in $\theta$ can be dropped.
\end{remark}

\begin{proof}
\begin{enumerate}[label=(\roman*)]
\item Combine Theorem~\ref{importante0}\ref{0i} applied to $F_\theta$,  Proposition~\ref{des}\ref{iii} and \ref{v}, and Proposition~\ref{pp}\ref{pp:claim1} to see that $\dist(X_0,\calS_{\theta}) \leq C_0$.
\item Argue as in claim~\ref{nuevo1:claim1} using Proposition~\ref{beta} instead of Theorem~\ref{importante0}\ref{0i}, and use the fact that $\nabla F_{\theta}$ is Lipschitz continuous with constant
\[
L+\frac{1}{\theta} \leq \frac{L+1}{\theta} \qforq \theta \in ]0,1].
\]
\item Combine Theorem~\ref{importante0}\ref{0ii} applied to $F_\theta$, Proposition~\ref{pp}\ref{pp:claim2} and Jensen's inequality. 
\end{enumerate}
\end{proof}

%% file: tex/sec-Appendix.tex
\appendix

\section{Auxiliary results}\label{aux}

\subsection{Deterministic results}
The following lemma is straightforward to prove. We omit the details.
\begin{lemma}\label{lim0}
Let $t_0>0$ and $g:[t_0,+\infty[\rightarrow \R_+$. Suppose that $\lim_{t\rightarrow \infty} g(t)$ exists and $\int_{t_0}^\infty \frac{g(s)}{s}ds<+\infty$. Then $\lim_{t\rightarrow \infty} g(t)=0$.
\end{lemma}
The next result is an adaptation of \cite[Proposition~2.3]{finitetime} to our specific context but under slightly less stringent assumptions.
\begin{lemma}[Comparison Lemma]\label{comparison}
Let $t_0\geq 0$ and $T> t_0$. Assume that $h:[t_0,+\infty[\rightarrow\R_+$ is measurable with $h \in \Lp^1([t_0,T])$ , that $\psi:\R_+\rightarrow\R_+$ is continuous and nondecreasing, $\varphi_0>0$ and the Cauchy problem
\begin{equation*}
\begin{cases}
\varphi'(t)=-\psi(\varphi(t)) + h(t) & \text{for almost all $t\in [t_0,T]$}\\
\varphi(t_0)=\varphi_0
\end{cases}
\end{equation*}
has an absolutely continuous solution $\varphi:[t_0,T]\rightarrow \R_+$. If a bounded from below lower semicontinuous function $\omega:[t_0,T]\rightarrow \R_+$ satisfies 
\[
\omega(t)\leq\omega(s)-\int_s^t \psi(\omega(\tau))d\tau + \int_s^t h(\tau)d\tau
\]
for $t_0\leq s < t \leq T$ and $\omega(t_0)=\varphi_0$, then 
\[
\omega(t)\leq \varphi(t)\quad \text{for $t\in [t_0,T]$}.
\]
\end{lemma}

\begin{theorem}[Egorov's Theorem] \cite[Chapter 3, Exercise 16]{rudin} \label{egorov}
 If $\mu(X)<\infty$ and $(f_t)_{t\in \R_+}$ is a family of real functions such that for all $x\in X$: 
 \begin{enumerate}
     \item $\lim_{t\rightarrow\infty} f_t(x)=f(x)$ and
     \item $t\mapsto f_t(x)$ is continuous.
 \end{enumerate} 
 Then, for every $\delta>0$, there exists a measurable set $E_{\delta}\subset X$, with $\mu(X\setminus E_{\delta})<\delta$, such that $(f_t)_{t\in\R_+}$ converges uniformly on $E_{\delta}$. 
\end{theorem}

\begin{lemma}\label{existenceof} 
Let $f:\R_+\rightarrow\R$ and $\liminf_{t\rightarrow\infty} f(t)\neq \limsup_{t\rightarrow\infty} f(t)$. Then there exists a constant $\alpha$, satisfying $\liminf_{t\rightarrow\infty} f(t)< \alpha<\limsup_{t\rightarrow\infty} f(t)$, such that for every $\beta>0$, we can define a sequence $(t_k)_{k\in\N}\subset\R$ such that 
\[
f(t_k)>\alpha,\quad t_{k+1}>t_k+\beta, \quad \forall k\in\N.
\]
\end{lemma}
\begin{proof}
 Since $\liminf_{t\rightarrow\infty} f(t)$ and $\limsup_{t\rightarrow\infty} f(t)$ are different real numbers, there exists $\alpha$ such that 
\[
\liminf_{t\rightarrow\infty} f(t)<\alpha<\limsup_{t\rightarrow\infty} f(t) .
\] 
Moreover, by definition of $\limsup$, there exists a sequence $(t_k)_{k\in\N}$ such that $\lim_{k\rightarrow\infty} t_k=\infty$ and $f(t_k)>\alpha$. Let $\beta>0$ and $n_0=0$, let us define recursively for $j\geq 1$, $n_j=\min\{n>n_{j-1}: t_n-t_{n_{j-1}}>\beta\}$. Let $j'\in\N$ be the first natural such that $n_{j'}=\infty$. This implies that for every $n>n_{j'-1}$, $t_n\leq\beta+t_{n_{j'-1}}<\infty$, a contradiction since $\lim_{n\rightarrow\infty}t_n=\infty$, then for every $j\in\N$, $n_j<\infty$. Thus, we can define $(t_{n_j})_{j\in\N}$ a subsequence of $(t_k)_{k\in\N}$ such that $\lim_{j\rightarrow\infty}t_{n_j}=\infty$ and for every $j\in\N$, $t_{n_{j+1}}-t_{n_j}>\beta$.
\end{proof}

\subsection{Stochastic results}
\begin{lemma}\label{limdelta}
Let $\delta>0,\Omega_{\delta}\in\calF$ such that $\Pro(\Omega_{\delta})\geq 1-\delta$ and $h:\Omega\times\R_+\rightarrow\R$ a stochastic process such that $\sup_{t\geq 0}\EE[h(\omega,t)^2]<\infty$. Then $$ \EE[h(\omega,t)\ind_{\Omega\setminus\Omega_{\delta}}]=\calO(\sqrt{\delta}).$$
\end{lemma}
\begin{proof}
Note that $\Pro(\Omega\setminus\Omega_{\delta})\leq \delta$ and
\begin{align*}
\EE[h(\omega,t)\ind_{\Omega\setminus\Omega_{\delta}}]&\leq \sqrt{\delta}\sqrt{\EE[h(\omega,t)^2]}\leq \sqrt{\delta}\sqrt{\sup_{t\geq 0}\EE[h(\omega,t)^2]},
\end{align*}
where we have used the Cauchy-Schwarz inequality for our first inequality. 
\end{proof}
\begin{corollary}\label{equality}
Let $X$ be the solution of \eqref{CSGD} under hypotheses \eqref{H0}, \eqref{H} on $f$ and $\sigma$, and  that $\sigma_{\infty}\in\Lp^2(\R_+)$. Then $h_1(\omega,t)=\frac{\dist(X(\omega,t),\calS)^2}{2}$ and $h_2(\omega,t)=f(X(\omega,t))-\min f$ satisfy the hypothesis of Lemma \ref{limdelta}, this means that there exists $C_d,C_f>0$: 
\begin{align*} 
\EE\pa{\frac{\dist(X(t),\calS)^2}{2}}-\EE\br{\frac{\dist(X(t),\calS)^2}{2}\ind_{\Omega_{\delta}}}&\leq C_d\sqrt{\delta}, \\
\EE\pa{f(X(t))-\min f}- \EE\br{(f(X(t))-\min f)\ind_{\Omega_{\delta}}}&\leq C_f\sqrt{\delta}. \end{align*}
\end{corollary}
\begin{proof}
Let $x^{\star}\in \calS$ be arbitrary. Using Proposition\ref{itoin} with $\hat{\phi}(x)=\frac{\dist(x,\calS)^2}{2}$, squaring it, and taking expectation, we obtain
\begin{align*}
    \EE\br{\frac{\dist^4(X(t),\calS)}{4}}&\leq 3\frac{\dist(X_0,\calS)^2}{4}+3\pa{\int_0^t\sigma_{\infty}^2(s)ds}^2\\
    &+3\EE\br{\pa{\int_0^t \langle \sigma^{\top}(s,X(s))(X(s)-P_\calS(X(s))),dW(s)\rangle}^2}\\
    &\leq 3\frac{\dist(X_0,\calS)^2}{4}+3\pa{\int_0^t\sigma_{\infty}^2(s)ds}^2+3\sup_{t\geq 0}\EE[\norm{X(t)-x^{\star}}^2]\br{\int_0^t  \sigma_{\infty}^2(s) ds}.
\end{align*}
Taking the supremum over $t\geq 0$, we obtain
\begin{align*}\sup_{t\geq 0}\EE\br{\pa{\frac{\dist(X(t),\calS)^2}{2}}^2}&\leq 3\frac{\dist(X_0,\calS)^2}{4}+3\pa{\int_0^{\infty}\sigma_{\infty}^2(s)ds}^2\\
&+3\sup_{t\geq 0}\EE[\norm{X(t)-x^{\star}}^2]\br{\int_0^{\infty}  \sigma_{\infty}^2(s) ds}\eqdef C_d<\infty.
\end{align*}
In the above estimation we used that $\sigma_{\infty}\in\Lp^2(\R_+)$ and $\sup_{t\geq 0}\EE[\norm{X(t)-x^{\star}}^2]<\infty$ by Theorem~\ref{converge2}\ref{acota}.

\smallskip

On the other hand, using Proposition \ref{itoin} with $\tilde{\phi}(x)=f(x)-\min f$, squaring it, and taking expectation, we obtain
\begin{align*}
    \EE\br{[f(X(t)-\min f]^2}&\leq 3[f(X_0)-\min f]^2+\frac{3L}{2}\pa{\int_0^t\sigma_{\infty}^2(s)ds}^2\\
    &+3\EE\br{\pa{\int_0^t \langle \sigma^{\top}(s,X(s))(\nabla f(X(s))),dW(s)\rangle}^2}\\
    &\leq 3[f(X_0)-\min f]^2+\frac{3L}{2}\pa{\int_0^t\sigma_{\infty}^2(s)ds}^2\\
    &+3L^2\sup_{t\geq 0}\EE[\norm{X(t)-x^{\star}}^2]\br{\int_0^t  \sigma_{\infty}^2(s) ds}.
\end{align*}
Taking the supremum over $t\geq 0$, we obtain
\begin{align*}\sup_{t\geq 0}\EE\br{[f(X(t)-\min f]^2}&\leq 3[f(X_0)-\min f]^2+\frac{3L}{2}\pa{\int_0^{\infty}\sigma_{\infty}^2(s)ds}^2\\
    &+3L^2\sup_{t\geq 0}\EE[\norm{X(t)-x^{\star}}^2]\br{\int_0^{\infty}  \sigma_{\infty}^2(s) ds}\eqdef C_f<\infty.
\end{align*}
\end{proof}

Let us consider the Stochastic Differential Equation:
\begin{align}\label{SDE1}
\begin{cases}
\begin{aligned}
dX(t) &= F(t,X(t))dt+G(t,X(t))dW(t), \quad t\geq 0, \\
X(0)&=X_0,
\end{aligned}
\end{cases}
\end{align}
where $F:\R_+\times\R^d\rightarrow\R^d$, $G:\R_+\times\R^d\rightarrow\R^{d\times m}$ are measurable functions 
and $W$ is a $\calF_t$-adapted $m$-dimensional Brownian Motion.
\begin{theorem} (See \cite[Theorem~5.2.1]{oksendal_2003}, \cite[Theorem~2.4.1]{mao}) \label{teoexistencia} 
Let $F: \R_+\times\R^d\rightarrow\R^d$ and $G:\R_+\times\R^d\rightarrow \R^{d\times m}$ be measurable functions satisfying, for every $T>0$: 
\begin{equation}\label{lips}
\norm{F(t,x)-F(t,y)}+\norm{G(t,x)-G(t,y)}_F\leq C_1\norm{x-y},\quad \forall x,y\in\R^d, \forall t\in [0,T],
\end{equation}
for some constant $C_1\geq 0$. Then \eqref{SDE1} has a unique solution $X\in S_d^l$, for every $l\geq 2$.
\end{theorem}
\begin{proof}
Condition \eqref{lips} implies that there exists $C_2\geq 0$ such that
\[
\norm{F(t,x)}+\norm{G(t,x)}_F\leq C_2(1+\norm{x}), \quad \forall x\in\R^d, \forall t\in [0,T].
\]
These are the hypotheses of \cite[Theorem~5.2.1]{oksendal_2003} to ensure the existence and uniqueness of the solution $X\in S_d^2$ of \eqref{SDE1}. Moreover, condition \eqref{lips} implies the existence of $C_3\geq 0$ such that 
\begin{equation}\label{condit}
\dotp{x}{F(t,x)}+\norm{G(t,x)}_F^2\leq C_3(1+\norm{x}^2) \quad \forall x\in\R^d, \forall t\in [0,T].
\end{equation}
Thus \eqref{condit} is the necessary inequality to use \cite[Lemma~3.2]{gig} and deduce that $X\in S_d^l$, for every $l\geq 2$. 
\end{proof}

\subsection{On martingales}

\begin{theorem}\cite{doob}\label{doob}
Let $(M_t)_{t\geq 0}:\Omega\rightarrow\R$ be a continuous martingale such that $\sup_{t\geq 0} \EE\pa{|M_t|^p}<\infty$ for some $p>1$. Then there exists a random variable $M_{\infty}$ such that $\EE\pa{|M_{\infty}|^p}<\infty$ and $\lim_{t\rightarrow\infty} M_t= M_{\infty}$ a.s.
\end{theorem}

\begin{theorem} \label{impp} \cite[Theorem 1.3.9]{mao}
 Let $\{A_t\}_{t\geq 0} $ and $\{U_t\}_{t\geq 0} $ be two continuous adapted increasing processes with $A_0=U_0=0$ a.s. Let $\{M_t\}_{t\geq 0} $ be a real valued continuous local martingale with $M_0=0$ a.s. Let $\xi$ be a nonnegative $\calF_0$-measurable random variable. Define 
\[
X_t=\xi+A_t-U_t+M_t \qforq t\geq 0.
\]
If $X_t$ is nonnegative and $\lim_{t\rightarrow\infty} A_t<\infty$ a.s., then a.s. $\lim_{t\rightarrow\infty} X_t$ exists and is finite, and $\lim_{t\rightarrow\infty} U_t<\infty$.
\end{theorem}